\documentclass[12pt]{amsart}
\usepackage[all]{xypic}
\usepackage{tikz}
\usetikzlibrary{arrows}
\usetikzlibrary{decorations.markings}

\usepackage{graphicx}
\usepackage{epsf}
\usepackage{verbatim} 
\usepackage{amsmath}
\usepackage{amsfonts}
\usepackage{amssymb}
\usepackage{mathrsfs}
\usepackage{amsthm}
\usepackage{newlfont}

 \newtheorem{thm}{Theorem}
 
 \newtheorem{cor}[thm]{Corollary}

 \newtheorem{lem}[thm]{Lemma}
 
 \newtheorem{prop}[thm]{Proposition}
 
 \theoremstyle{definition}
 \newtheorem{defn}[thm]{Definition}
 
 \theoremstyle{definition}
 \newtheorem{notn}[thm]{Notation}

 \theoremstyle{remark}
 \newtheorem{rem}[thm]{Remark}
 
 \theoremstyle{definition}
 \newtheorem{example}[thm]{Example}

 \numberwithin{thm}{section}
 \numberwithin{equation}{section}

 \newcommand{\ab}{\mathrm{ab}}
 
 \newcommand{\Hom}{\mathrm{Hom}}
 
 \newcommand{\Spec}{\mathrm{Spec}}

 \newcommand{\End}{\mathrm{End}}
 \newcommand{\Pic}{\mathrm{Pic}}

 \newcommand{\Gal}{\mathrm{Gal}}
 \newcommand{\GL}{\mathrm{GL}}
 \newcommand{\PGL}{\mathrm{PGL}}

 \newcommand{\rank}{\mathrm{rank}}

 \newcommand{\tor}{\mathrm{tor}}
 
 \newcommand{\Stab}{\mathrm{Stab}}
 
 \newcommand{\new}{\mathrm{new}}

 \newcommand{\Tr}{\mathrm{Tr}}

 \renewcommand{\mod}{\mathrm{mod}}
 
 \newcommand{\fp}{\mathfrak p}
 \newcommand{\fq}{\mathfrak q}
 \newcommand{\fr}{\mathfrak r}
 \newcommand{\fn}{\mathfrak n}
 \newcommand{\fm}{\mathfrak m}
 \newcommand{\fd}{\mathfrak d}

 \newcommand{\fE}{\mathfrak E}
 
 \newcommand{\fI}{\mathfrak I}
 \newcommand{\cO}{\mathcal{O}}

 \renewcommand{\cH}{\mathcal{H}}

 \newcommand{\cE}{\mathcal{E}}
 
 \newcommand{\cC}{\mathcal{C}}

 \newcommand{\cG}{\mathcal{G}}
 \newcommand{\cJ}{\mathcal{J}}
 \newcommand{\cI}{\mathcal{I}}
 \newcommand{\cT}{\mathcal{T}}

\newcommand{\sT}{\mathscr{T}}

 \newcommand{\gm}{\mathbb{G}}

 \newcommand{\C}{\mathbb{C}}
 \newcommand{\F}{\mathbb{F}}
 \newcommand{\Q}{\mathbb{Q}}
 \newcommand{\W}{\mathbb{W}}
 \newcommand{\Z}{\mathbb{Z}}

 \newcommand{\p}{\mathbb{P}}
 \newcommand{\T}{\mathbb{T}}
 
 \newcommand{\N}{\mathbb{N}}

 \newcommand{\G}{\Gamma}
 \newcommand{\To}{\longrightarrow}
 \newcommand{\bs}{\setminus}
 \newcommand{\Fi}{F_\infty}
 \newcommand{\bD}{\overline{\Delta}}

 \newcommand{\bG}{\overline{\Gamma}}

\begin{document}

\title{On the Eisenstein ideal over function fields}

\author{Mihran Papikian}
\address{Department of Mathematics, Pennsylvania State University, University Park, PA 16802, U.S.A.}
\email{papikian@psu.edu}
\author{Fu-Tsun Wei}
\address{Institute of Mathematics, Academia Sinica, 6F, Astronomy-Mathematics Building, No. 1, Sec. 4, Roosevelt Road, Taipei 10617, Taiwan}
\email{ftwei@math.sinica.edu.tw}

\thanks{The first author was supported in part by the Simons Foundation.} 
\subjclass[2010]{11G09, 11G18, 11F12}
\keywords{Drinfeld modular curves; Cuspidal divisor group; Eisenstein ideal.}

\dedicatory{Dedicated to Winnie Li}



\begin{abstract}
We study the Eisenstein ideal of Drinfeld modular curves of small levels, and the 
relation of the Eisenstein ideal to the cuspidal divisor group and the component groups of Jacobians of Drinfeld modular curves. 
We prove that the characteristic of the function field is an Eisenstein prime number 
when the level is an arbitrary non square-free ideal of $\F_q[T]$ not equal to a square of a prime. 
\end{abstract}


\maketitle


\section{Introduction} 

The Eisenstein ideal for modular curves over $\Q$ was introduced by Mazur in his seminal paper \cite{Mazur}, 
and since then the Eisenstein ideal has become an indispensable tool in various problems related to modular curves, 
modular Jacobians, modular Galois representations, etc. The problem to develop the theory of Eisenstein ideals 
for Drinfeld modular curves was suggested by Mazur, already in the introduction of \cite{Mazur}. 
The first attempt to develop this theory was made by Tamagawa \cite{Tamagawa}, but more 
comprehensive results were obtained by P\'al \cite{Pal}. Both \cite{Tamagawa} and \cite{Pal} assume 
that the level is prime. In \cite{PW}, in connection with the problem 
of Jacquet-Langlands isogenies over function fields, we examined the Eisenstein ideal on Drinfeld modular curves 
whose level is a product of two distinct primes. We discovered that some of the properties 
of the Eisenstein ideal in that case are quite different from its prime level counterpart. 
In this paper we continue our study of the Eisenstein ideal for non-prime levels, and its 
relation to the cuspidal divisor group and the component groups of Jacobians of Drinfeld modular curves. 
Our goal here is to compute everything explicitly when the level is small, and from this make some 
predictions about the behaviour of the Eisenstein ideal in general. 

\vspace{0.1in}

Let $\F_q$ be a finite field with $q$ 
elements, where $q$ is a power of a prime number $p$. Let $A=\F_q[T]$ 
be the ring of polynomials in indeterminate $T$ with coefficients 
in $\F_q$, and $F=\F_q(T)$ be the rational function field. 
The degree map $\deg: F\to \Z\cup \{-\infty\}$, which associates 
to a non-zero polynomial its degree in $T$ and $\deg(0)=-\infty$, defines 
a norm on $F$ by $|a|:=q^{\deg(a)}$. The corresponding place of $F$ 
is usually called the \textit{place at infinity}, and is denoted by $\infty$; it plays a role similar to the archimedean place of $\Q$. 
We also define a norm and degree on the ideals of $A$ by $|\fn|:=\#(A/\fn)$ and $\deg(\fn):=\log_q|\fn|$. 
Let $\Fi$ denote the completion of $F$ at $\infty$, and $\C_\infty$ denote the 
completion of an algebraic closure of $\Fi$. Let $\Omega:=\C_\infty - \Fi$ be the \textit{Drinfeld half-plane}.  

Let $\fn\lhd A$ be a non-zero ideal. The level-$\fn$ \textit{Hecke congruence subgroup} of $\GL_2(A)$ is 
$$
\G_0(\fn):=\left\{\begin{pmatrix} a & b \\  c & d\end{pmatrix}\in \GL_2(A)\ \bigg|\ c\equiv 0\ \mod\  \fn \right\}.   
$$
Let $\T(\fn)$ be the $\Z$-algebra generated by the Hecke operators 
$T_\fm$, $\fm\lhd A$, acting on the group $\cH_0(\fn, \Z)$  
of $\Z$-valued $\G_0(\fn)$-invariant 
cuspidal harmonic cochains on the Bruhat-Tits tree $\sT$ of $\PGL_2(\Fi)$; see Section \ref{sAL} 
for the definitions. 
The \textit{Eisentein ideal} $\fE(\fn)$ of $\T(\fn)$ is the ideal generated by the elements 
$$
\left\{T_\fp-|\fp|-1\ \big|\ \fp \text{ is prime}, \fp \nmid \fn\right\}.
$$
(For some alternative ways of defining this ideal see $\S$\ref{ssAltDef}.) The quotient ring $\T(\fn)/\fE(\fn)$ 
is finite (Lemma \ref{lemTE0}), and constitutes the main object of study of this paper. 
In some sense, $\T(\fn)/\fE(\fn)$  encodes congruences between cuspidal harmonic cochains and 
Eisenstein series. 

\begin{defn}
Let $G$ be a $\T(\fn)$-module. We say that $g\in G$ is \textit{Eisenstein} if $T_\fp g=(|\fp|+1)g$ 
for all prime $\fp\nmid \fn$. The Eisenstein elements form a submodule of $G$. We will 
denote this submodule by $G[\fE(\fn)]$. (It is clear that if $g$ is Eisenstein, then it is annihilated by 
all elements of  $\fE(\fn)$, which justifies the notation.) We say that $G$ is \textit{Eisenstein} if $G=G[\fE(\fn)]$. 
\end{defn}

The group $\G_0(\fn)$ acts on $\Omega$ via linear fractional transformations. 
Drinfeld proved in \cite{Drinfeld} that the quotient $\G_0(\fn)\bs \Omega$ 
is the space of $\C_\infty$-points of an affine curve $Y_0(\fn)$ defined over $F$,  
which is a moduli space of rank-$2$ Drinfeld modules.  
The unique smooth projective curve over $F$ containing $Y_0(\fn)$ as an 
open subvariety is denoted by $X_0(\fn)$. 
The Hecke algebra $\T(\fn)$ naturally acts on the Jacobian $J_0(\fn)$ 
of $X_0(\fn)$. This action functorially extends to the N\'eron model of $J_0(\fn)$, 
hence $\T(\fn)$ also acts on the component groups of $J_0(\fn)$. 

It is well-known that the component groups of classical modular Jacobians $J_0(N)$ are Eisenstein. 
This was proved by Ribet in the semistable reduction case \cite{RibetCGSS}, and 
by Edixhoven in general \cite{EdixhovenECG}. It is more-or-less clear that the arguments in \cite{RibetCGSS} and 
\cite{EdixhovenECG} can be transferred to the function fields setting (although this is not in published literature), so 
it is very likely that the component groups of Drinfeld modular Jacobians $J_0(\fn)$ at finite primes are Eisenstein. 
On the other hand, in addition to the primes dividing $\fn$, $J_0(\fn)$ also has bad (purely toric) reduction at $\infty$, 
and it is not hard to construct examples where the component group $\Phi_\infty(\fn)$ of $J_0(\fn)$ at $\infty$ 
is not Eisenstein; see Examples \ref{exmP2} and \ref{example810}.  
In this paper we prove the following:

\begin{thm}\label{thmMAIN1}
Assume $\deg(\fn)=3$. Then $\Phi_\infty(\fn)$ is Eisenstein, and there is an isomorphism of $\T(\fn)$-modules 
$\T(\fn)/\fE(\fn)\cong \Phi_\infty(\fn)$. 
\end{thm}

Note that $\deg(\fn)=3$ is the smallest degree for which $\T(\fn)\neq 0$. When $\deg(\fn)=3$, 
the $\Z$-rank of $\T(\fn)$ is equal to $q$ (resp. $q-1$) if $\fn$ is square-free (resp. not square-free). 
Up to an affine transformation $T\mapsto aT + b$ with
$a \in \F_q^\times$ and $b \in \F_q$, there are 5 different cases, namely 
\begin{enumerate}
\item $\fn=T^3$;
\item $\fn=T^2(T-1)$;
\item $\fn$ is irreducible;
\item $\fn=T\fp$, where $\fp$ is irreducible of degree $2$;
\item $\fn=T(T-1)(T-c)$, where $c\in \F_q$, $c\neq 0,1$ (here we must have $q>2$). 
\end{enumerate}

In Section \ref{sCDG}, for $\deg(\fn)=3$ we compute $\Phi_\infty(\fn)$, the cuspidal divisor group $\cC(\fn)$ of $J_0(\fn)$, 
and the canonical homomorphism $\wp_\infty: \cC(\fn)\to \Phi_\infty(\fn)$ arising from the N\'eron mapping property. 
With above numbering of cases, the results are the following: 
\begin{enumerate}
\item $\cC(\fn)\overset{\wp_\infty}{\cong}\Phi_\infty(\fn)\cong \Z/q^2\Z$. 
\item $\cC(\fn)\overset{\wp_\infty}{\cong}\Phi_\infty(\fn)\cong \Z/q(q^2-1)\Z$. 
\item $\cC(\fn)\cong \Phi_\infty(\fn)\cong \Z/(q^2+q+1)\Z$, but $\wp_\infty$ is not necessarily an isomorphism 
$$
\xymatrix{0 \ar[r] & \Z/(3, q-1)\Z\ar[r]  &\cC(\fn) \ar[r]^-{\wp_\infty}  & \Phi_\infty(\fn)\ar[r] & \Z/(3, q-1)\Z \ar[r] & 0.}
$$
\item 
$$ \cC(\fn)\cong \Z/(q+1)\Z\oplus \Z/(q^2+1)\Z, $$ $$\Phi_\infty(\fn)\cong \Z/(q^2+1)(q+1)\Z,$$
$$
\xymatrix{0 \ar[r] & \Z/(2, q-1)\Z\ar[r]  &\cC(\fn) \ar[r]^-{\wp_\infty}  & \Phi_\infty(\fn)\ar[r] & \Z/(2, q-1)\Z \ar[r] & 0.}
$$
\item 
$$ \cC(\fn)\cong \Z/(q+1)\Z\oplus \Z/(q+1)\Z\oplus \Z/(q-1)(q+1)\Z,$$ 
$$\Phi_\infty(\fn)\cong \Z/(q+1)\Z\oplus \Z/(q+1)\Z\oplus \Z/(q-1)^2(q+1)\Z,$$
$$
\xymatrix{0 \ar[r] & \cC(\fn) \ar[r]^-{\wp_\infty}  & \Phi_\infty(\fn)\ar[r] & \Z/(q-1)\Z \ar[r] & 0.}
$$
\end{enumerate}

\begin{rem}
In fact, (4) is a result from \cite{PapikianJNT}, so Section \ref{sCDG} contains only the 
calculations for the other cases. Also, (3) is a result of Gekeler \cite[$\S$6]{GekelerCDG}, but our argument is  
somewhat different.  
\end{rem}

We proved Theorem \ref{thmMAIN1} in \cite{PW} for the case (4). In Section \ref{sec4}, we adapt the 
argument from \cite{PW} to the other cases. The outline of this argument is the following.  
First, we show that $\Phi_\infty(\fn)$ is Eisenstein. Next, we 
show that there is a $\T(\fn)$-equivariant surjective homomorphism $\T(\fn)\to \Phi_\infty(\fn)$. This 
implies that there is a $\T(\fn)$-equivariant surjection $\T(\fn)/\fE(\fn)\to \Phi_\infty(\fn)$. Finally, we give an upper bound on the 
order of $\T(\fn)/\fE(\fn)$ which matches the order of $\Phi_\infty(\fn)$, so the previous surjection is an isomorphism. 

To carry out the strategy outlined above, in Section \ref{sHA}, we prove some preliminary results 
about the Hecke algebra $\T(\fn)$, which might be of independent interest. Let $\T(\fn)^0$ be 
the subalgebra of $\T(\fn)$ generated by the Hecke operators $T_\fm$ with $\fm$ coprime to $\fn$. 
\begin{thm}\label{thmMAIN4}
Assume $\deg(\fn)=3$. 
\begin{itemize}
\item[(i)] In all cases, except (5), $\T(\fn)=\T(\fn)^0$. 
\item[(ii)] There is a natural isomorphism of $\T(\fn)$-modules $$\Hom(\cH_0(\fn, \Z), \Z)\cong \T(\fn).$$
\item[(iii)] The Hecke operators $\{T_\fp\ |\ \deg(\fp)=1\}$ span $\T(\fn)$ over $\Z$.  
\end{itemize}
\end{thm}

\begin{rem}
The Hecke algebra $\T(\fn)$ for $\deg(\fn)=3$ was studied by Gekeler in \cite{GekelerKleinem}, and (iii) 
is implicitly contained there. 
\end{rem}

An unexpected consequence of Theorem \ref{thmMAIN1} is that $\T(\fn)\neq \T(\fn)^0$ in case (5), 
although the index $[\T(\fn):\T(\fn)^0]$ is finite. (The index is finite because there are no ``old forms'' 
in $\cH_0(\fn, \Z)$ when $\deg(\fn)=3$.) In Section \ref{sIAL}, we prove an analogue of 
a theorem of Atkin and Lehner for $\cH_0(\fn,\Z)$, and deduce from this the following restriction on the index: 

\begin{thm}\label{thmMAIN6} Assume
$\fn=T(T-1)(T-c)$, where $c\in \F_q$, $c\neq 0,1$. 
If a prime number $\ell$ divides the order of $\T(\fn)/\T(\fn)^0$, then $\ell$ divides $q(q+1)$. Conversely, if $\ell$ 
divides $q+1$, then $\ell$ divides the order of $\T(\fn)/\T(\fn)^0$. 
\end{thm}

In \cite{Mazur}, as one of the first applications of his theory of Eisenstein ideal, Mazur 
proved that the rational torsion subgroup $J_0(N)(\Q)_\tor$ coincides with the cuspidal divisor group when $N$ is prime.  
The analogue of this result for Drinfeld Jacobians $J_0(\fn)$ of prime level was proved by P\'al \cite{Pal}. 
In Section \ref{sRT}, as an application of Theorem \ref{thmMAIN1}, we prove 

\begin{thm}\label{thmMAIN7} Let $\cT(\fn)$ be the torsion subgroup of the group of $F$-rational points of $J_0(\fn)$. 
The cuspidal divisor group is rational over $F$ and coincides with $\cT(\fn)$ 
for $\fn=T^3$ and $\fn=T^2(T-1)$. 
\end{thm}

\begin{rem}
In \cite{PW}, we proved that $\cT(\fn)=\cC(\fn)$ also for case (4), and, the equality 
$\cT(\fn)=\cC(\fn)$ in case (3) is a special case of P\'al's result. Thus, only case (5) remains open. 
\end{rem}

We say that a prime number $\ell\in \N$ is an \textit{Eisenstein prime number} for $\fn$ if $\ell$ divides the order of $\T(\fn)/\fE(\fn)$. 
When $\deg(\fn) = 3$, we observe that $p$ (= the characteristic of $F$) is an Eisenstein prime number if $\fn$ is not square-free. 
Moreover, there is a cuspidal divisor which is $F$-rational and has order divisible by $p$.  
In Section \ref{sEPN}, we prove that this is actually a special case of the following:

\begin{thm}\label{thmMAIN8}
Assume $\fn$ is divisible by $\fp^2$ for some prime $\fp\lhd A$, but $\fn\neq \fp^2$. Then 
there is a cuspidal divisor in $\cC(\fn)$ which is Eisenstein, rational over $F$, and has order divisible by $p$. 
This implies that $p$ is an Eisenstein prime number for $\fn$. 
\end{thm}

\begin{rem}
In general, $\cC(\fn)$ is neither Eisenstein nor rational over $F$; see Example \ref{exmP2}. 
This example also shows that the assumption $\fn\neq \fp^2$ is necessary in Theorem \ref{thmMAIN8}. 
\end{rem}

Let $\T(\fn)'$ be the quotient of $\T(\fn)^0$ through which $\T(\fn)^0$ acts on the new quotient $J_0(\fn)^\new$ of $J_0(\fn)$. 
Let $\fE(\fn)'$ be the ideal of $\T(\fn)'$ generated by the images of elements $T_\fp-|\fp|-1$, 
where $\fp\nmid \fn$ is prime. In \cite{PalIJNT}, P\'al proved that $p$ does not divide the order of 
$\T(\fn)'/\fE(\fn)'$ if $\fn$ is square-free. If $\deg(\fn)=3$, then $J_0(\fn)^\new= J_0(\fn)$, so $\T(\fn)'=\T(\fn)^0$. 
 Moreover, by Theorem \ref{thmMAIN4}, $\T(\fn)^0=\T(\fn)$ when $\fn$ is not square-free. Thus, $\T(\fn)'=\T(\fn)$ 
 when $\fn$ is not square-free, and Theorem \ref{thmMAIN8} shows that $p$ divides $\T(\fn)'/\fE(\fn)'$. 
 In particular, the assumption that $\fn$ is square-free is necessary in \cite{PalIJNT}. 
 
 The results of this paper raise two interesting questions, which we are not able to answer for the present:
 
 \begin{enumerate}
 \item[(a)] Is it true that for general non-zero $\fn\lhd A$, there is an isomorphism 
 $$
 \T(\fn)/\fE(\fn)\cong \Phi_\infty[\fE(\fn)]?
 $$
 \end{enumerate}
 Theorem \ref{thmMAIN1} shows that this is true for $\deg(\fn)=3$, and Theorem \ref{thmEispnot} 
 shows that this is true if $\fn$ is prime. Also, Example \ref{exmP2} shows 
 that this is true for $q=2$ and $\fn=(T^2+T+1)^2$. 
 \begin{enumerate}
 \item[(b)] Is it true that $p$ is \textbf{not} an Eisenstein prime number for $\fn=\fp^2$, where $\fp\lhd A$ 
 is an arbitrary prime ideal? 
 \end{enumerate}
 Example \ref{exmP2} shows 
 that this is true for $q=2$ and $\fn=(T^2+T+1)^2$. Theorem \ref{thmMAIN8} shows that $p$ is an Eisenstein prime number 
 for any non square-free $\fn\neq \fp^2$. 

 \subsection*{Acknowledgements} We gratefully dedicate this paper to Winnie Li for her advice and encouragement 
 over the years. Part of this work was carried out during the Workshop on Function Field Arithmetic held at the 
 Nesin Mathematics Village in 2014. We thank the organizers of the workshop for their invitation, and the staff 
 of the village for creating a welcoming atmosphere.  


\section{Preliminaries}\label{sAL} 

\subsection{Notation} 
Besides $\infty$, the other places of $F$ are in bijection with the non-zero prime ideals of $A$.  
Given a place $v$ of $F$, we denote by $F_v$ the completion of $F$ at $v$, by $\cO_v$ 
the ring of integers of $F_v$, and by $\F_v$ the residue field of $\cO_v$. 
We fix $\pi:=T^{-1}$ as a uniformizer of $\cO_\infty$. 

Let $R$ be a commutative ring with unity. We denote by $R^\times$ the group of 
multiplicative units of $R$. Let 
$\GL_n(R)$ be the group of $n\times n$ matrices over $R$ whose determinant is in $R^\times$, and $Z(R)\cong R^\times$ 
the subgroup of  $\GL_n(R)$ consisting of scalar matrices. 

If $X$ is a scheme over a base $S$ and $S'\to S$ any base change, $X_{S'}$ 
denotes the pullback of $X$ to $S'$. If $S'=\Spec(R)$ is affine, 
we may also denote this scheme by $X_R$. By $X(S')$ we mean the $S'$-rational points 
of the $S$-scheme $X$, and again, if $S'=\Spec(R)$, we may also denote this set by $X(R)$. 

Given an abelian group $H$ and an integer $n$, $H[n]$ is the kernel of multiplication by 
$n$ in $G$. For a prime number $\ell$, $H_\ell$ is the $\ell$-primary component of $H$. 

Given an ideal $\fn\lhd A$, by abuse of notation, we denote by the same symbol the unique 
monic polynomial in $A$ generating $\fn$. It will always be clear from the context 
in which capacity $\fn$ is used; for example, if $\fn$ appears in a matrix, column vector, or a 
polynomial equation, then the monic polynomial is implied. 
The prime ideals $\fp\lhd A$ are always assumed to be non-zero. 
The notation $\fm\parallel \fn$ means that $\fm$ divides $\fn$ and $\gcd(\fm, \fn/\fm)=1$. 


\subsection{Harmonic cochains} 
Let $G$ be an oriented connected graph in the sense of Definition 1 of $\S$2.1 in \cite{SerreT}. 
We denote by $V(G)$ and $E(G)$ its sets of vertices and edges, respectively. 
For an edge $e\in E(G)$, let $o(e)$, $t(e)\in V(G)$ and $\bar{e}\in E(G)$ be its 
origin, terminus and inversely oriented edge, respectively. In particular, $t(\bar{e})=o(e)$ 
and $o(\bar{e})=t(e)$. We will assume that 
for any $v\in V(G)$ the number of edges with $t(e)=v$ is finite,  
and $\bar{e}\neq e$ for any $e\in E(G)$.  A \textit{path} in $G$ is a 
sequence of edges $\{e_i\}_{i\in I}$ indexed by a set $I$ where $I=\Z$, $I=\N$ or 
$I=\{1,\dots, m\}$ for some $m\in \N$ such that $t(e_i)=o(e_{i+1})$ for every 
$i, i+1\in I$.  We say that the path is \textit{without backtracking} if $e_i\neq \bar{e}_{i+1}$ 
for every $i, i+1\in I$. We say that the path without backtracking $\{e_i\}_{i\in \N}$ is a \textit{half-line} 
if for every vertex $v$ of $G$ there is at most one index $n\in \N$ 
such that $v=o(e_n)$. 

Let $\G$ be a group acting on a graph $G$. We say that $\G$ acts with \textit{inversion} if there is
$\gamma\in \G$ and $e\in E(G)$ such that $\gamma e=\bar{e}$. If
$\G$ acts without inversion, then we have a natural quotient graph
$\G\bs G$ such that $V(\G\bs G)=\G\bs V(G)$ and
$E(\G\bs G)=\G\bs E(G)$, cf. \cite[p. 25]{SerreT}.

\begin{defn}
Let $R$ be a commutative ring with unity. An $R$-valued \textit{harmonic cochain} on $G$ is a 
function $f: E(G)\to R$ that satisfies 
\begin{itemize}
\item[(i)] $$f(e)+f(\bar{e})=0\quad \text{for all $e\in E(G)$},$$
\item[(ii)] 
$$\sum_{\substack{e\in E(G) \\ t(e)=v}} f(e)=0\quad \text{for all $v\in V(G)$}.$$
\end{itemize}
Denote by $\cH(G, R)$ the group of $R$-valued harmonic cochains on $G$.
\end{defn}

The most important graphs in this paper are the Bruhat-Tits tree $\sT$ of $\PGL_2(\Fi)$, and the 
quotients of $\sT$. We recall the definition and introduce some notation for later use. 
The sets of vertices $V(\sT)$ and edges $E(\sT)$ are the cosets $\GL_2(\Fi)/Z(\Fi)\GL_2(\cO_\infty)$ 
and $\GL_2(\Fi)/Z(\Fi)\cI_\infty$, respectively, where $\cI_\infty$ is the Iwahori group:
$$
\cI_\infty=\left\{\begin{pmatrix} a & b\\ c & d\end{pmatrix}\in \GL_2(\cO_\infty)\ \bigg|\ c\in \pi\cO_\infty\right\}. 
$$
The matrix $\begin{pmatrix} 0 & 1\\ \pi & 0\end{pmatrix}$ 
normalizes $\cI_\infty$, so the multiplication from the right by this matrix on $\GL_2(\Fi)$ 
induces an involution on $E(\sT)$; this involution is $e\mapsto \bar{e}$. 
The matrices 
\begin{equation}\label{eq-setM}
E(\sT)^+=\left\{\begin{pmatrix} \pi^k & u \\ 0 & 1\end{pmatrix}\ \bigg|\
\begin{matrix} k\in \Z\\ u\in \Fi,\ u\ \mod\ \pi^k\cO_\infty\end{matrix}\right\}
\end{equation}
are in distinct left cosets of $\cI_\infty Z(\Fi)$, and there is a disjoint decomposition  
$$
E(\sT)=E(\sT)^+\bigsqcup E(\sT)^+\begin{pmatrix} 0 & 1\\ \pi & 0\end{pmatrix}. 
$$
We call the edges in $E(\sT)^+$ \textit{positively oriented}. 
The group $\GL_2(\Fi)$ naturally acts on $E(\sT)$ by left multiplication. 
This induces an action on the group of $R$-valued functions on $E(\sT)$: 
for a function $f$ on $E(\sT)$ and $\gamma\in \GL_2(\Fi)$ we define the function $f|\gamma$ on $E(\sT)$ by 
$(f|\gamma)(e)=f(\gamma e)$. 
It is clear from the definition that $f|\gamma$ is harmonic if $f$ is harmonic. 

Let $\G$ be a subgroup of $\GL_2(\Fi)$ which acts on $\sT$ without inversions. 
Denote by $\cH(\sT, R)^\G$ the subgroup of $\G$-invariant harmonic cochains, i.e.,  
$f|\gamma=f$ for all $\gamma\in \G$.
It is clear that $f\in \cH(\sT, R)^\G$ defines a function $f'$ on the quotient graph $\G\bs\sT$, and 
$f$ itself can be uniquely recovered from this function: if $e\in E(\sT)$ maps to $\tilde{e}\in E(\G\bs \sT)$ under the quotient map, 
then $f(e)=f'(\tilde{e})$. The conditions of harmonicity (i) and (ii) can be formulated 
in terms of $f'$ as follows. Since $\G$ acts without inversion, (i) is equivalent to 
\begin{itemize}
\item[(i$'$)] 
$$
f'(\tilde{e})+f'(\bar{\tilde{e}})=0\quad \text{for all $\tilde{e}\in E(\G\bs\sT)$}.
$$
\end{itemize}
Let $v\in V(\sT)$ and $\tilde{v}\in V(\G\bs\sT)$ be its image. The stabilizer group 
$$\G_v=\{\gamma\in \G\ |\ \gamma v=v\}$$  acts on the set $\{e\in E(\sT)\ |\ t(e)=v\}$, 
and the orbits correspond to $$\{\tilde{e}\in E(\G\bs \sT)\ |\ t(\tilde{e})=\tilde{v}\}.$$ 
Let $\G_e:=\{\gamma\in \G\ |\ \gamma e=e\}$; clearly $\G_e$ is a subgroup of $\G_{t(e)}$. 
The \textit{weight} of $e$ 
$$
w(e):=[\G_{t(e)}:\G_e]
$$
is the length of the orbit corresponding to $e$. Since $w(e)$ depends only on its image $\tilde{e}$ in $\G\bs \sT$, 
we can define $w(\tilde{e}):=w(e)$. Note that 
$\sum_{t(\tilde{e})=\tilde{v}} w(\tilde{e}) = q+1$. 
(In general, $w(e)$ depends on the orientation, i.e., $w(e)\neq w(\bar{e})$.) 
With this notation, condition (ii) is equivalent to 
\begin{itemize}
\item[(ii$'$)] 
$$
\sum_{\substack{\tilde{e}\in E(\G\bs\sT) \\ t(\tilde{e})=\tilde{v}}} w(\tilde{e})f'(\tilde{e})=0\quad 
\text{for all $\tilde{v}\in V(\G\bs\sT)$}.
$$
\end{itemize}

\begin{defn} The group of $R$-valued \textit{cuspidal 
harmonic cochains} for $\G$, denoted $\cH_0(\sT, R)^\G$, is the 
subgroup of $\cH(\sT, R)^\G$ consisting of functions 
which have compact support as functions on $\G\bs\sT$, i.e., functions 
which have value $0$ on all but finitely many edges of $\G\bs\sT$.  
\end{defn}

To simplify the notation, we put 
\begin{align*}
\cH(\fn, R)&:=\cH(\sT, R)^{\G_0(\fn)}\\
\cH_0(\fn, R)&:=\cH_0(\sT, R)^{\G_0(\fn)} \\
\cH_{00}(\fn, R)&:=\text{ the image of }\cH_0(\fn, \Z)\otimes R \text{ in } \cH_0(\fn, R). 
\end{align*}

It is known that the quotient 
graph $\G_0(\fn)\bs \sT$ is the edge disjoint union 
$$
\G_0(\fn)\bs \sT = (\G_0(\fn)\bs \sT)^0\cup \bigcup_{s\in \G_0(\fn)\bs \p^1(F)} h_s
$$
of a finite graph $(\G_0(\fn)\bs \sT)^0$ with a finite number of half-lines $h_s$, called \textit{cusps}. 
The cusps are in bijection with the orbits of the natural (left) action of $\G_0(\fn)$ on $\p^1(F)$; cf. \cite[(2.6)]{GR}. 
It is clear that $f\in \cH(\fn, R)$ is cuspidal if and only if it eventually vanishes on each $h_s$. 
It is also clear that if $R$ is flat over $\Z$, then $\cH_0(\fn, R)=\cH_{00}(\fn, R)$. 
On the other hand, it is easy to construct examples where this equality does not hold; cf. \cite[$\S$1.1]{PW}. 

One can show that $\cH_0(\fn, \Z)$ and $\cH(\fn, \Z)$ are finitely generated free 
$\Z$-modules of rank $g(\fn)$ and  $g(\fn)+c(\fn)-1$, respectively, where $g(\fn)$ 
is the genus of $X_0(\fn)$ and $c(\fn)$ is the number of cusps. 

\subsection{Hecke operators and the Eisenstein ideal}\label{ssHOEI}
Fix a non-zero ideal $\fn\lhd A$. Given a non-zero ideal $\fm\lhd A$, define 
an $R$-linear transformation of the space of $R$-valued functions on $E(\sT)$ by 
\begin{equation}\label{eqDefTm}
f|T_\fm=\sum f|\begin{pmatrix} a & b \\ 0 & d\end{pmatrix},
\end{equation}
where the sum is over $a,b,d\in A$ such that $a,d$ are monic, $(ad)=\fm$, $(a)+\fn=A$, and $\deg(b)< \deg(d)$. 
This transformation is the \textit{$\fm$-th Hecke operator}. Following a common convention, 
for a prime divisor $\fp$ of $\fn$ we sometime write $U_\fp$ instead of $T_\fp$. 

\begin{prop}\label{propHOm} The Hecke operators preserve $\cH_0(\fn, R)$, and satisfy the 
recursive formulas: 
\begin{align*}
T_{\fm\fm'}&= T_\fm T_{\fm'}\quad \text{if}\quad  \fm+\fm'=A,\\
T_{\fp^i} &= T_{\fp^{i-1}}T_\fp-|\fp|T_{\fp^{i-2}}\quad \text{if}\quad  \fp\nmid \fn, \\
T_{\fp^i} &= T_\fp^i\quad \text{if}\quad  \fp| \fn. 
\end{align*}
\end{prop}
\begin{proof} The group-theoretic proofs of the analogous statement for the Hecke operators 
acting on classical modular forms work also in this setting; cf. \cite[$\S$4.5]{Miyake}. 
\end{proof}

\begin{defn}\label{defHA}
Let $\T(\fn)$ be the commutative $\Z$-subalgebra of $\End_\Z(\cH_0(\fn, \Z))$ 
generated by all Hecke operators. Let $\T(\fn)^0$ be the subalgebra of $\T(\fn)$ generated 
by the Hecke operators $T_\fm$ with $\fm$ coprime to $\fn$. 
\end{defn}

\begin{defn}
For every ideal $\fm\parallel \fn$, let $W_\fm$ be any matrix of the form 
\begin{equation}\label{ALmatrix}
\begin{pmatrix} a\fm & b \\ c\fn & d\fm \end{pmatrix} 
\end{equation}
such that $a,b,c,d, \in A$, and the ideal generated by $\det(W_\fm)$ in $A$ is $\fm$. 
\end{defn}

It is not hard to check that for $f\in \cH_0(\fn, \Z)$, $f|W_\fm$ does not depend on the choice 
of the matrix for $W_\fm$ and  $f|W_\fm\in \cH_0(\fn, \Z)$. Moreover, 
as $\Z$-linear endomorphisms of $\cH(\fn, \Z)$, $W_\fm$'s satisfy 
\begin{equation}\label{eqWs}
W_{\fm_1}W_{\fm_2}=W_{\fm_3}, \quad \text{where} \quad \fm_3=
\frac{\fm_1\fm_2}{\mathrm{gcd}(\fm_1, \fm_2)^2}. 
\end{equation}
Therefore, the matrices $W_\fm$ acting on $\cH_0(\fn, \Z)$ 
generate an abelian group $\W\cong (\Z/2\Z)^s$, called the group of \textit{Atkin-Lehner involutions}, 
where $s$ is the number of distinct prime divisors of $\fn$. 

\begin{lem}\label{lemUW} Let $f\in \cH_0(\fn, \Z)$. 
\begin{enumerate}
\item If $\fp^2\mid \fn$, then $f|U_\fp\in \cH_0(\fn/\fp, \Z)$. 
\item If $\fp\parallel \fn$, then $f|(U_\fp+W_\fp)\in \cH_0(\fn/\fp, \Z)$. 
\end{enumerate}
\end{lem}
\begin{proof} For $\fm\lhd A$, denote $D_\fm=\begin{pmatrix} 1 & 0 \\ 0 & \fm\end{pmatrix}$ and 
$$
\G_0(\fn,\fm)=\left\{\begin{pmatrix} a & b \\  c & d\end{pmatrix}\in \GL_2(A)\ \big|\ c\in \fn, b\in \fm \right\}. 
$$
Suppose $\fm\mid \fn$. Since 
$$
\begin{pmatrix} 1 & 0 \\  0 & \fm\end{pmatrix} \begin{pmatrix} a & b \\  c & d\end{pmatrix} 
= \begin{pmatrix} a & b/\fm \\  c\fm & d\end{pmatrix} \begin{pmatrix} 1 & 0 \\  0 & \fm\end{pmatrix},
$$
we see that if $f$ is a $\G_0(\fn)$-invariant function on $E(\sT)$, then $f|D_\fm$ is $\G_0(\fn/\fm, \fm)$-invariant.  

Let $\fp$ be a prime dividing $\fn$, and let $S_\fp$ be a set of right coset representatives of $\G_0(\fn/\fp, \fp)$ in $\G_0(\fn/\fp)$. 
Since $f|D_\fp$ is $\G_0(\fn/\fp, \fp)$-invariant, the function 
$$
\sum_{\gamma\in S_\fp} (f|D_\fp) |\gamma
$$
is $\G_0(\fn/\fp)$-invariant. 

If $\fp$ divides $\fn/\fp$, then one checks that a set of right coset representatives for $\G_0(\fn/\fp, \fp)$ 
in $\G_0(\fn/\fp)$ is given by $\left\{\begin{pmatrix} 1 & b \\  0 & 1\end{pmatrix}\big| \deg(b)<\deg(\fp)\right\}$. The first claim then follows from 
$$
\sum_{\deg(b)<\deg(\fp)} f|D_\fp \begin{pmatrix} 1 & b \\  0 & 1\end{pmatrix} = f|U_\fp. 
$$

If $\fp$ does not divide $\fn/\fp$, then a set of right coset representatives for $\G_0(\fn/\fp, \fp)$ 
in $\G_0(\fn/\fp)$ is given by $\left\{\begin{pmatrix} 1 & b \\  0 & 1\end{pmatrix}\big| \deg(b)<\deg(\fp)\right\}\cup B$, where 
$B\in \G_0(\fn/\fp)$ is any matrix of the form $\begin{pmatrix} \alpha \fp & 1 \\  \delta\fn/\fp & 1\end{pmatrix}$. Now 
$$
\sum_{\deg(b)<\deg(\fp)} f|D_\fp \begin{pmatrix} 1 & b \\  0 & 1\end{pmatrix} + f|D_\fp B= f|U_\fp+f|W_\fp  
$$
is $\G_0(\fn/\fp)$-invariant, which proves the second claim. 
\end{proof}

\begin{cor}\label{cor1.8}
Suppose $\fp$ is a prime dividing $\fn$ and $\deg(\fn/\fp)\leq 2$. We have the following equalities of operators acting on $\cH_0(\fn, \Z)$:
\begin{enumerate}
\item If $\fp^2 | \fn$, then $U_\fp=0$. 
\item If $\fp^2\nmid\fn$, then $U_\fp=-W_\fp$. 
\end{enumerate}
\end{cor}
\begin{proof} 
If $\deg(\fn/\fp)\leq 2$, then $\cH_0(\fn/\fp, \Z)=0$; cf. \cite{GN}. Now apply Lemma \ref{lemUW}.  
\end{proof}


\begin{defn}\label{defEI} The \textit{Eisenstein ideal} $\fE(\fn)$ of $\T(\fn)$ is the 
ideal generated by the elements $T_\fp-(|\fp|+1)$ for all prime $\fp\nmid \fn$. Let $\fE(\fn)^0$ 
be the ideal of $\T(\fn)^0$ generated by the same elements. 
Denote 
$$
\cE_{00}(\fn, R)=\cH_{00}(\fn, R)[\fE(\fn)]
$$ 
\end{defn}

\begin{lem}\label{lemTE0} Let $\fn\lhd A$ be a non-zero ideal. 
\begin{enumerate}
\item $\T(\fn)/\fE(\fn)$ is finite. 
\item $\T(\fn)^0/\fE(\fn)^0$ is finite and cyclic. 
\item The exponent of $\T(\fn)/\fE(\fn)$ divides the order of $\T(\fn)^0/\fE(\fn)^0$. 
\end{enumerate}
\end{lem}
\begin{proof} 
To simplify the notation, we will omit the level $\fn$. Since $\T/\fE$ is finitely generated, to prove the first statement it is enough to 
prove that $(\T/\fE)\otimes \Q \cong (\T\otimes \Q)/(\fE\otimes \Q)=0$. The pairing 
$$
(\T\otimes \Q) \times \cH_0(\fn, \Q)\to \Q
$$
obtained from (\ref{GPairing}) by extension of scalars is perfect, so if $(\T\otimes \Q)/(\fE\otimes \Q)\neq 0$, then $\cE_{00}(\fn, \Q)\neq 0$. 
This implies that for any prime $\fp\nmid \fn$ the operator $\eta_\fp:=T_\fp-|\fp|-1$ acting on $\cH_0(\fn, \Q)$ 
has non-trivial kernel, contradicting \cite[p. 366]{GekelerIJM}, according to which $\eta_\fp$ is invertible. (This is a consequence 
of Drinfeld's reciprocity \cite{Drinfeld} and the Weil conjectures.)  

To prove the second statement, note that from the definition of $\T^0$ and Proposition \ref{propHOm} 
it follows that the $\Z$-algebra $\T^0$ is generated by $T_\fp$'s with $\fp\nmid \fn$ prime. Since $T_\fp$ 
is congruent to an element of $\Z$ modulo $\fE^0$, the natural inclusion of $\Z$ into $\T^0$ 
induces a surjection $\Z\to \T^0/\fE^0$. Hence to prove the second statement it is enough to show that $\fE^0$ 
contains a non-zero integer. Fix some $\fp\nmid \fn$. Let $x^m+a_{m-1}x^{n-1}+\cdots+a_1x +a_0$ 
be the characteristic polynomial of $\eta_\fp$ acting on $\cH_0(\fn, \Z)$. Since $\eta_\fp$ is invertible, $a_0\neq 0$. 
By the Cayley-Hamilton theorem, 
$$
-a_0=\eta_\fp^m+a_{m-1}\eta_\fp^{n-1}+\cdots+a_1\eta_\fp\in \fE^0. 
$$

Finally, to prove the last statement, note that $\T/\fE$ is a $\T^0$-module. Let $N$ be the order of $\T(\fn)^0/\fE(\fn)^0$. 
Since $\fE^0$ obviously annihilates 
$\T/\fE$, we get that $N\in \fE^0$ annihilates $\T/\fE$. 
\end{proof}


\subsection{Fourier expansion}\label{ssFE} Let $f\in \cH_0(\fn,\C)$. Let $\eta: \Fi\to \C^\times$ be the character 
\begin{equation}\label{eq-eta}
\eta: \sum a_i\pi^i \mapsto \eta_0(\Tr(a_1)), 
\end{equation}
where $\Tr: \F_q\to \F_p$ is the trace and $\eta_0$ is any nontrivial character of $\F_p$. 
For $\fm\lhd A$, the \textit{$\fm$-th Fourier coefficient} $f^\ast(\fm)$ of $f$ is 
$$
f^\ast(\fm) = q^{-1-\deg(\fm)}\sum_{u\in (\pi)/(\pi^{2+\deg(\fm)})} 
f\left(\begin{pmatrix}\pi^{2+\deg(\fm)} & u \\ 0 &1\end{pmatrix}\right)\eta(-\fm u).  
$$ 
It is easy to show that 
\begin{equation}\label{eqfast1}
f^\ast(1)=-f\left(\begin{pmatrix}\pi^2 & \pi \\ 0 &1\end{pmatrix}\right), 
\end{equation}
and the Fourier coefficients $f^\ast(\fm)$ of $f\in \cH_0(\fn, \Z)$ lie in $\Z[p^{-1}]$; cf. \cite[pp. 42-43]{Analytical}.

The cuspidal harmonic cochain $f$ can be uniquely recovered from its Fourier coefficients via the \textit{Fourier expansion} 
$$
f \left(\begin{pmatrix} \pi^k & u \\ 0 &1\end{pmatrix}\right) = 
\sum_{0\leq j\leq k-2} q^{-k+2+j}\sum_{\deg(\fm)=j}f^\ast(\fm)\nu(\fm u) 
$$
with $\nu(v)=-1$ if $v$ has a term of order $\pi$ in its $\pi$-expansion, and $\nu(v)=q-1$ if it has no term of order $\pi$; see 
\cite[$\S$2]{Improper}. 

The action of Hecke operators on the Fourier expansion is given by the formula, cf. \cite[Lem. 3.2]{Pal}, 
$$
(f|T_\fm)^\ast(\fr) = \sum_{\substack{a\ \mathrm{monic}\\ a|\gcd(\fm, \fr)\\  (a)+\fn=A}}\frac{|\fm|}{|a|}\cdot f^\ast\left(\frac{\fr \fm}{a^2}\right). 
$$
In particular, 
$$
(f|T_\fm)^\ast(1)=|\fm|f^\ast(\fm). 
$$


\subsection{Drinfeld modular curves}\label{ssDMC} 
The Drinfeld half-plane 
$$\Omega=\p^1(\C_\infty)-\p^1(\Fi)=\C_\infty-\Fi$$ 
has a natural structure of a smooth connected rigid-analytic space 
over $\Fi$; see \cite[$\S$1]{GR}. 
The group $\G_0(\fn)$ acts on $\Omega$ via linear 
fractional transformations:
$$
\begin{pmatrix} a & b \\ c & d \end{pmatrix}z=\frac{az+b}{cz+d}. 
$$
There is a smooth affine algebraic curve $Y_0(\fn)$ defined over $F$ whose analytification over $\Fi$ is 
isomorphic to the quotient $\G_0(\fn)\bs \Omega$; cf. \cite[Prop. 6.6]{Drinfeld}. Let $X_0(\fn)$ be the smooth 
projective model of $Y_0(\fn)$. The curve $X_0(\fn)$ is the \textit{Drinfeld modular curve} corresponding to $\G_0(\fn)$. 
 The points $X_0(\fn)(\C_\infty)-Y_0(\fn)(\C_\infty)$, called the \textit{cusps} of $X_0(\fn)$, 
 are in natural bijection with the cusps of $\G_0(\fn)\bs \sT$. Moreover, the genus of $X_0(\fn)$ is equal to $\rank_\Z \cH_0(\fn, \Z)$. 

 The curve $Y_0(\fn)$ is the generic fibre of the coarse moduli scheme for the functor which associates to an $A$-scheme $S$ 
 the set of isomorphism classes of pairs $(\phi, C_\fn)$, where $\phi$ is a Drinfeld $A$-module of rank $2$ over $S$ 
 and $C_\fn\cong A/\fn$ is an ``$\fn$-cyclic subgroup'' of $\varphi$; we refer to \cite{Drinfeld}, \cite{Uber}, \cite{GR} for the details. 

 Let $\fm$ be a divisor of $\fn$. There is a functorial morphism 
\begin{equation}\label{eqLLM}
 X_0(\fn)\to X_0(\fm),
\end{equation}
which in terms of the moduli problem is given by $(\phi, C_\fn)\mapsto (\phi, C_{\fm})$, 
where $C_{\fm}$ is the $\fm$-cyclic subgroup of $C_\fn$. This morphism is defined over $F$ and 
maps the cusps of $X_0(\fn)$ to the cusps of $X_0(\fm)$, cf. \cite{KM}. 
 
The Jacobian variety $J:=J_0(\fn)$ of $X_0(\fn)$ has bad reduction at $\infty$ and at the primes dividing $\fn$. 
Let $v$ be a place of bad reduction of $J$, and let $\cJ$ denote the N\'eron model of $J$ over $\cO_v$. 
Let $\cJ^0$ denote the relative connected component of the identity of $\cJ$, that is, the largest open 
subscheme of $\cJ$ containing the identity section which has connected fibres.  
The \textit{group of connected components} (or \textit{component group}) of $J$ at $v$ 
is $\Phi_v:=\cJ_{\F_v}/\cJ_{\F_v}^0$. The homomorphism 
$\wp_v: J(F_v)\to \Phi_v$ obtained from the composition 
$$
\wp_v: J(F_v)=\cJ(\cO_v)\to \cJ_{\F_v}(\overline{\F}_v)\to \Phi_v
$$
will be called the \textit{canonical specialization map}. (Of course, the component groups depend on $\fn$, 
which we will omit from notation.)

The Hecke operator $T_\fm$ may also be defined as a correspondence on $X_0(\fn)$, so 
$T_\fm$ induces an endomorphism of the Jacobian variety 
 $J:=J_0(\fn)$ of $X_0(\fn)$. The $\Z$-subalgebra of $\End_F(J)$ generated 
by all Hecke endomorphisms is canonically isomorphic to $\T(\fn)$; this is a 
consequence of Drinfeld's reciprocity law \cite[Thm. 2]{Drinfeld}.  
The endomorphisms $T_\fm$ of $J$ canonically extend 
to $\cJ$ and preserve $\cJ^0$, hence act on $\Phi_v$. 


\subsection{Grothendieck's monodromy pairing}\label{ssGMP}
Fix a non-zero ideal $\fn\lhd A$, and denote $\G=\G_0(\fn)$. 
Let $e\in E(\G\bs \sT)$ and $\tilde{e}\in E(\sT)$ be any edge mapping to $e$ under the quotient map. 
The stabilizer group 
$$
\Stab_{\G}(\tilde{e}):=\{\gamma\in \G\ |\ \gamma\tilde{e}=\tilde{e}\}
$$
is finite, contains the scalar matrices $Z(\F_q)$, and does not depend on the choice of $\tilde{e}$ mapping to $e$. Denote 
\begin{equation}\label{eqn(e)}
n(e)=\# \Stab_{\G}(\tilde{e})/\F_q^\times. 
\end{equation}
This is an integer which is equal to $1$ for most edges in $(\G\bs\sT)^0$. It is clear that $n(e)$ does not depend on the orientation of $e$. 
Define a pairing on $\cH_0(\fn, \Z)$ by 
\begin{equation}\label{eqPIP}
(f, g)=\sum_{e\in E(\G\bs\sT)^+}f(e)g(e)n(e)^{-1}. 
\end{equation}
Since $f$ and $g$ are cuspidal, all but finitely many terms of this sum are zero, so the pairing is well-defined. 
It is clear that $(\cdot, \cdot)$ is symmetric and positive-definite. It is also $\Z$-valued, as follows from \cite[(5.7)]{GR}. 

\begin{rem}\label{remPIP}
The Haar measure on $\GL_2(\Fi)$ induces a push-forward measure on $E(\G\bs \sT)$, which, up to a scalar multiple, 
is equal to $n(e)^{-1}$; cf. \cite[(4.8)]{GR}. 
One can show that (\ref{eqPIP}) agrees with the restriction to $\cH_0(\fn, \Z)$ of the 
Petersson scalar product if one interprets $\cH_0(\fn, \C)$ as a space 
of automorphic forms; see \cite[5.7]{GR}.
\end{rem}

The existence of rigid-analytic 
uniformization of Drinfeld modular curves over $\Fi$ implies that $X_0(\fn)_{\Fi}$ is a Mumford curve. Thus,   
the Jacobian $J:=J_0(\fn)$ has split purely toric reduction at $\infty$, i.e., $\cJ^0_{\F_\infty}$ is a split algebraic 
torus over $\F_\infty$. It follows from the theory of Mumford curves that the character group of the torus $\cJ^0_{\F_\infty}$ 
can be identified with the simplicial homology group $H_1(\G\bs\sT, \Z)$.  

Let $\varphi\in H_1(\G\bs\sT, \Z)$, regarded as a $\G$-invariant function $\varphi: E(\sT)\to \Z$. 
Then $\varphi^\ast: e\mapsto n(e)\varphi(e)$ is a well-defined element of $\cH_0(\fn, \Z)$, and $\varphi\mapsto \varphi^\ast: 
H_1(\G\bs\sT, \Z)\to \cH_0(\fn, \Z)$ is an isomorphism by \cite{GN}. Via these isomorphisms, the pairing (\ref{eqPIP}) 
is Grothendieck's monodromy pairing with respect to the canonical principal polarization on $J$; see \cite[$\S$4]{PapikianAIF}. 
Therefore, by a theorem of Grothendieck, there is a short exact sequence (cf. \textit{loc. cit.})
\begin{equation}\label{eqmopa-inf}
0\To \cH_0(\fn, \Z)\xrightarrow{f\mapsto (f, \cdot)} \Hom(\cH_0(\fn, \Z), \Z)\To \Phi_\infty\To 0.
\end{equation}

Let $\theta$ be the canonical principal polarization on $J$. Let $\alpha\mapsto \alpha^\dag$ be the 
Rosati involution of $\End(J)$ induced by $\theta$. By the N\'eron mapping property, an endomorphism $\alpha$ of $J$
extends to an endomorphism of $\cJ$ which preserves $\cJ^0$. Hence $\alpha$ canonically acts on $\Phi_\infty$. 
Moreover, considering $\cH_0(\fn, \Z)$ as the character group of 
$\cJ^0_{\F_\infty}$, we get an endomorphism of $\cH_0(\fn, \Z)$, which we again denote by $\alpha$. The adjoint 
of $\alpha$ with respect to the monodromy pairing (\ref{eqPIP}) is $\alpha^\dag$ (cf. \cite[$\S$3.3]{PapikianCM}), i.e.,  
$$
(\alpha f, g)=(f, \alpha^\dag g)\quad \text{for all }f, g\in \cH_0(\fn, \Z). 
$$
Let $\alpha$ act on $\Hom(\cH_0(\fn, \Z), \Z)$ through its action on the first argument. If we make $\alpha$ act 
on the first term of the exact sequence (\ref{eqmopa-inf}) as $\alpha^\dag$, then the whole sequence 
becomes $\alpha$-equivariant. 

For $t\in \T(\fn)$, considered as an element of $\End(J)$, we have (cf. \cite[p. 444]{RibetLL})
$$
W_\fn t W_\fn=t^\dag. 
$$
(This reflects the fact that the adjoint of $t$ with respect to the Petersson inner product is $W_\fn t W_\fn$.)
Since the Atkin-Lehner involution $W_\fn$ commutes with $T_\fp\in \T(\fn)$ for any $\fp\nmid \fn$, the exact sequence 
(\ref{eqmopa-inf}) is $\T(\fn)^0$-equivariant.  

\begin{rem}
The existence of the exact sequence (\ref{eqmopa-inf}) was deduced by Gekeler
as a consequence of the rigid-analytic uniformization of $J_0(\fn)$ constructed in \cite{GR}, without using Grothendieck's result; 
see Corollary 2.11 in \cite{Analytical}. The $\T(\fn)^0$-equivariance 
of (\ref{eqmopa-inf}) then follows from the $\T(\fn)^0$-equivariance of this uniformization; see \cite[$\S$9]{GR}.
\end{rem}


\section{The cuspidal divisor group}\label{sCDG}

The \textit{cuspidal divisor group} $\cC(\fn)$ 
is the subgroup of $J_0(\fn)$ generated by the classes of divisors $[c]-[c']$, where $c, c'$  
run through the set of cusps of $X_0(\fn)$. The cuspidal divisor group is finite \cite{GekelerIJM}. 

The cusps of $X_0(\fn)$ are in bijection with the orbits 
of the action of $\G_0(\fn)$ on 
$$
\p^1(F)=\p^1(A)=\left\{\begin{pmatrix} a \\ b\end{pmatrix}\ \big|\ a, b\in A, \gcd(a, b)=1, a \text{ is monic}\right\}, 
$$
where $\G_0(\fn)$ acts on $\p^1(F)$ from the left as on column vectors. 
The rational cusps are those cusps which are defined over $F$, i.e., which give $F$-rational points on $X_0(\fn)$. 

\begin{lem}\label{lemCD1} Let 
$
\fn=\fp_1^{r_1}\cdots \fp_s^{r_s}$
be the prime factorization of $\fn$. Put 
$$
\kappa(\fn)=\prod_{i=1}^s \left(|\fp_i|^{\lfloor r_i/2\rfloor}+|\fp_i|^{\lfloor (r_i-1)/2\rfloor}\right). 
$$
\begin{enumerate}
\item[(i)] Every cusp of $X_0(\fn)$ has a representative $\begin{pmatrix} a \\ b\end{pmatrix}$ where $a, b\in A$ are 
monic, $b| \fn$, and $\mathrm{gcd}(a, \fn)=1$. Two such representatives $\begin{pmatrix} a \\ b\end{pmatrix}$ 
and $\begin{pmatrix} a' \\ b'\end{pmatrix}$ represent the same cusp of $X_0(\fn)$ if and only if $b=b'$ 
and $\alpha a'=a$ modulo $\tilde{b}:=\mathrm{gcd}(b, \fn/b)$ for some $\alpha\in \F_q^\times$. 
\item[(ii)] The total number of cusps of $X_0(\fn)$ is $2^s+\frac{\kappa(\fn)-2^s}{q-1}$. 
\item[(iii)] The cusps with the same $b$ are conjugate over $F$. In particular, $\begin{pmatrix} a \\ b\end{pmatrix}$ 
is rational if and only if $\deg(\tilde{b})\leq 1$, or $q=2$ and $\tilde{b}=T^2+T$. 
\item[(iv)]
If $\fp_i$ is a prime divisor of $\fn$ of degree one, put 
$$
t_i=
\begin{cases}
0 & r_i=1, \\
1 & r_i=1, \\
2 & r_i\geq 3. 
\end{cases}
$$
Let $u=t_1\cdot t_2$ if ($q=2$ and $\fp_1=T$, $\fp_2=T-1$ 
are divisors of $\fn$), and $u=0$ otherwise. 
Then the number of rational cusps of $X_0(\fn)$ is 
$$
2^s+2^{s-1}\sum t_i+2^{s-2}u,
$$
where the sum is over the prime divisors of degree one of $\fn$. 
\end{enumerate}
\end{lem}
\begin{proof} See \cite[$\S$6]{Invariants} and \cite[Prop. 1]{SchweizerHE}. 
\end{proof}

\begin{lem}\label{lemWcusp} Let $\fm$ be a non-trivial ideal dividing $\fn$ with $\mathrm{gcd}(\fm, \fn/\fm)=1$, and let $W_\fm$ 
denote the corresponding Atkin-Lehner involution. 
Let $\begin{pmatrix} a \\ b\end{pmatrix}$ and $\begin{pmatrix} a' \\ b'\end{pmatrix}$ be 
cusps of $X_0(\fn)$ with $$\begin{pmatrix} a' \\ b'\end{pmatrix}=W_\fm \begin{pmatrix} a \\ b\end{pmatrix}.$$ 
Then $$\mathrm{gcd}(b, \fm)\cdot \mathrm{gcd}(b', \fm)=\fm \quad \text{ and } \quad 
\mathrm{gcd}(b, \fn/\fm)= \mathrm{gcd}(b', \fn/\fm).$$
\end{lem}
\begin{proof}
See \cite[Lem. 8]{SchweizerHE}. 
\end{proof}

\begin{cor}\label{cor2.5} Assume $\deg(\fn)=3$. 
\begin{enumerate}
\item[(i)] The cusps of $X_0(\fn)$ are in bijection with the monic divisors of $\fn$ via 
$$
\fd|\fn \mapsto [\fd]:=\text{$\G_0(\fn)$-orbit of }\begin{pmatrix} 1 \\ \fd\end{pmatrix}. 
$$
\item[(ii)] All cusps of $X_0(\fn)$ are rational. 
\item[(iii)] For $\fm\parallel \fn$ and prime $\fp\parallel \fn$, we have 
\begin{align*} 
&W_\fn[\fd]=[\fn/\fd], \\ 
&W_\fm[\fn]=[\fn/\fm], \\ 
&W_\fp[\fd]=\begin{cases}
[\fd/\fp] &\text{if $\fp$ divides $\fd$};\\
[\fd\fp] &\text{otherwise}. 
\end{cases} 
\end{align*}  
\end{enumerate}
\end{cor}
\begin{proof}
(i) and (ii) follow from Lemma \ref{lemCD1}, and (iii) follows from Lemma \ref{lemWcusp}. 
\end{proof}

\begin{notn} Assume $\deg(\fn)=3$. 
Let $[\infty]:=[\fn]$. For a monic divisor $\fd$ of $\fn$, let $c_\fd:=[\fd]-[\infty]\in \cC(\fn)$. Note that 
$c_\fd$'s generate $\cC(\fn)$.  
\end{notn}

\begin{lem}\label{lem2.4}
If $\deg(\fn) = 3$, then $\cC(\fn)$ is Eisenstein. 
\end{lem}
\begin{proof}
By Lemma~\ref{lemCD1} (i) and Corollary~\ref{cor2.5} (i), we observe that for each cusp $c$ of $X_0(\fn)$,
$$\begin{pmatrix}1&u\\0&\fp \end{pmatrix} c = 
\begin{pmatrix} \fp&0\\0&1\end{pmatrix} c = c, \quad \text{for any $u$ and $\fp$ with $\fp \nmid \fn$, $\deg(u)<\deg(\fp)$}. 
$$
This implies that $T_\fp c=(|\fp|+1) c$, hence $\cC(\fn)$ is Eisenstein. 
\end{proof}

\begin{rem}\label{rem3.6}
(1) In general, $\cC(\fn)$ may not be Eisenstein (see Example \ref{exmP2}), but $\cC(\fn)$ is 
Eisenstein if $\fn$ is square-free (see the proof of Lemma 3.1 in \cite{PW}). \\ 
(2) Let $F(\fn)$ be the field of $\fn$-division points of the Carlitz module. It is an abelian extension of $F$ 
with Galois group $(A/\fn)^\times$. Let $F_+(\fn)$ be the maximal subfield of $F(\fn)$ in which $\infty$ 
totally splits. By a result of Gekeler (cf. \cite[Thm. 4.6]{Invariants}) the cusps of $X_0(\fn)$ are 
$F_+(\fn)$-rational points. Hence $\cC(\fn)$ is a subgroup of $F_+(\fn)$-rational points of $J_0(\fn)$. 
\end{rem}

For the rest of this section we compute the group structure of $\cC(\fn)$ for $\fn$ of degree $3$. Our strategy 
for doing this is the following. 
Take a prime divisor $\fp$ of $\fn$ and consider the functorial morphism $X_0(\fn)\to X_0(\fn/\fp)$ discussed in $\S$\ref{ssDMC}. 
Since $\deg(\fn/\fp)\leq 2$ and 
$[\infty]$ is rational, $X_0(\fn/\fp)$ is isomorphic to the projective line over $F$. 
The pullbacks of principal divisors on $X_0(\fn/\fp)$ supported at the cusps give relations between the cuspidal divisors. 
Next, by \cite{SchweizerHE} and \cite{Uber}, $X_0(\fn)$ is hyperelliptic (if $\deg(\fn)=3$), and the Atkin-Lehner 
involution $W_\fn$ is the hyperelliptic involution. This gives another morphism $X_0(\fn)\to \p^1_{F}\cong X_0(\fn)/W_\fn$, 
from which one deduces an extra relation for the cuspidal divisors. The relations that we obtain from these 
calculations give the generators of $\cC(\fn)$ and upper bounds on the orders of these generators. 

Next, we compute $\Phi_\infty$. One can compute the component groups by using a theorem of Raynaud \cite[Thm. 9.6/1]{NM},  
assuming the structure of the special fibre of a regular model of $X_0(\fn)$ over $\cO_\infty$ is known.  
Such a model can be obtained from the rigid-analytic uniformization of this curve; see \cite[$\S$4.2]{PapikianAIF}. 
More precisely, the structure of the special fibre $X_0(\fn)_{\F_\infty}$ of the minimal regular model of $X_0(\fn)$ 
over $\cO_\infty$ can be deduced from the structure of the quotient graph $\G_0(\fn)\bs \sT$ and the stabilizers of the edges; 
we refer to \cite[$\S$5.2]{PapikianJNT} for a more detailed explanation and a carefully worked out example. 
The quotient graphs $\G_0(\fn)\bs \sT$ are described in Section \ref{sHA}. 

Finally, we compute the canonical specialization $\wp_\infty: \cC(\fn)\to \Phi_\infty$ by using the relative position 
of the cusps on $\G_0(\fn)\bs \sT$. (Note that by Remark \ref{rem3.6} $\cC(\fn)$  
is $F_\infty$-rational, so $\wp_\infty$ is defined on $\cC(\fn)$.) This gives lower bounds for the orders of generators of $\cC(\fn)$, 
which turn out to match the previous upper bounds. 


\subsection{$\fn$ is irreducible} In this case $X_0(\fn)$ has two cusps, $[1]$ and $[\infty]$. 
Hence $\cC(\fn)$ is cyclic, generated by $c_1$. By \cite[(5.11)]{Uber}, $\cC(\fn)=\Z/(q^2+q+1)\Z$. 

$X_0(\fn)$ has a regular model over $\cO_\infty$ whose special fibre is depicted in Figure \ref{Fig1}. 

\begin{figure}[h]
\begin{tikzpicture}[scale=0.7, inner sep=.3mm]
\draw[thick, rounded corners=20pt]
(1,1) -- (4,4) -- (6,2) -- (8,4) -- (10,2);
\draw[thick, rounded corners=20pt]
(1,5) -- (4,2) -- (6,4) -- (8,2) -- (10,4);
\draw[thick, dashed] (1.5, 1) .. controls (0,3) .. (1.5, 5);
\node at (2.2, 2) [label=right:$Z$] {};
\node at (2.2, 4) [label=right:$Z'$] {};
\end{tikzpicture}
\caption{$X_0(\fn)_{\F_\infty}$: $\fn$ irreducible}\label{Fig1}
\end{figure}

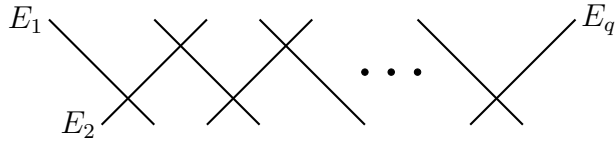
\begin{figure}[h]
\begin{tikzpicture}[scale=0.7, inner sep=.3mm]
\draw[thick] (0,2) -- (2,0); 
\draw[thick] (1,0) -- (3,2); 
\draw[thick] (2,2) -- (4,0); 
\draw[thick] (3,0) -- (5,2); 
\draw[thick] (4,2) -- (6,0); 
\node at (6, 1) [circle,draw,fill=black] {};
\node at (6.5, 1) [circle,draw,fill=black] {};
\node at (7, 1) [circle,draw,fill=black] {};
\draw[thick] (7,2) -- (9,0); 
\draw[thick] (8,0) -- (10,2); 
\node at (0, 2) [label=left:$E_1$] {};
\node at (1, 0) [label=left:$E_2$] {};
\node at (10, 2) [label=right:$E_q$] {};
\end{tikzpicture}
\caption{Chain of projective lines}\label{Fig2}
\end{figure}

Here $Z, Z'$ are projective lines defined over $\F_\infty\cong \F_q$ intersecting each other transversally in $q$ points. The dashed line is a chain 
of $q$ projective lines defined over $\F_q$ as in Figure \ref{Fig2}; here $E_1$ intersects $E_2$ transversally at an $\F_q$-rational point, 
$E_2$ intersects $E_3$, etc. Finally, $Z$ intersects the chain only at $E_q$ and $Z'$ intersects only $E_1$,  
both transversally.  The reductions of the cusps $[1]$ and $[\infty]$ lie on $E_1$ and $E_q$, respectively, 
away from the points of intersection of these lines with $Z',E_2$ and $Z,E_{q-1}$, respectively. All of this 
follows from $\S$\ref{ss31}. 

Let $B^0$ be the free abelian group with generators $z:=Z-Z'$ and $e_i:=E_i-Z'$, $1\leq i\leq q$. 
Then $\Phi_\infty$ is isomorphic to the quotient of $B^0$ by the relations coming from the intersection 
pairing of irreducible components $Z,Z',E_1, \dots, E_q$ with $X_0(\fn)_{\F_\infty}$; cf. \cite[$\S$4.2]{PapikianJNT}. 
These relations are the following:
$$
e_i=ie_1\quad  (1\leq i\leq q), \quad e_1=-qz, \quad e_q=(q+1)z. 
$$
This implies that $\Phi_\infty$ is generated by $z$ modulo a single relation $(q^2+q+1)z=0$, so $\Phi_\infty\cong \Z/(q^2+q+1)\Z$. 

Now, $\wp_\infty(c_1)=e_1-e_q=-qz-(q+1)z=-(2q+1)z=((q^2+q+1)-(2q+1))z=q(q-1)z$. Since $q$ is coprime to $q^2+q+1$, we get 
the exact sequence 
\begin{equation}\label{eqGekDocExm}
\xymatrix{0 \ar[r] & \Z/t\Z\ar[r]  &\cC(\fn) \ar[r]^-{\wp_\infty}  & \Phi_\infty\ar[r] & \Z/t\Z \ar[r] & 0,}
\end{equation}
$$
t=\mathrm{gcd}(q-1, q^2+q+1) =
\begin{cases}
3 & \text{if $q\equiv 1\ \mod\ 3$},\\
1 & \text{otherwise}.
\end{cases} 
$$ 
In particular, even though $\cC(\fn)$ and $\Phi_\infty$ are isomorphic as abelian groups, the canonical specialization map $\wp_\infty$ 
is not always an isomorphism. 



\subsection{$\fn=T^3$} The cusps of $X_0(\fn)$ are $[1], [T], [T^2], [\infty]$. 
Hence $\cC(\fn)$ is generated by $c_1$, $c_T$, and $c_{T^2}$. 
The canonical morphism $f: X_0(T^3)\to X_0(T^2)$ has degree $q$.  With slight abuse of notation, denote the cusps of $X_0(T^2)$ 
by $[1], [T], [T^2]$. One computes that $f([1])=[1]$ with ramification index $q$, $f([T])=[T]$ 
with ramification index $q$, $f([T^2])=[T^2]$ with ramification index $q-1$, and $f([\infty])=[T^2]$ 
with ramification index $1$; cf. \cite[p. 196]{Discriminant}. Since $X_0(T^2)\cong \p^1_F$, there 
is a function on $X_0(T^2)$ with divisor $[1]-[T^2]$, and a function with divisor $[T]-[T^2]$. Hence in $\Pic^0(X_0(\fn))$ we have 
\begin{align*}
0 &=f^\ast([1]-[T^2])=q[1]-((q-1)[T^2]+[\infty])=qc_1-(q-1)c_{T^2},\\ 
0 &=f^\ast([T]-[T^2])=q[T]-((q-1)[T^2]+[\infty])=qc_T-(q-1)c_{T^2}.
\end{align*}
Next,  by Corollary \ref{cor2.5}, $W_\fn([\infty])=[1]$ and $W_\fn([T])=[T^2]$. Let $P$ and $Q$ be the images of $[T]$ and $[\infty]$ 
under the quotient map $X_0(\fn)\to X_0(\fn)/W_\fn\cong \p^1_F$ of degree $2$. Pulling back the function with divisor $P-Q$ 
on $\p^1_F$, we get 
$$
0=[T]+[T^2]-[1]-[\infty]=c_T+c_{T^2}-c_1. 
$$
This relation, combined with the previous two, implies that $\cC(\fn)$ is generated by $c_T$ which has order dividing $q^2$. 
More precisely, $c_{T^2}=-qc_T$, $c_1=(q-1)c_T$, and $q^2c_T=0$. Note that this implies that $c_1$ also generates $\cC(\fn)$. 

The calculation of $\Phi_\infty$ and the specialization map $\wp_\infty$ is very similar to the case of irreducible $\fn$. 
In fact, $X_0(T^3)$ has a regular model over $\cO_\infty$ whose special fibre has the same structure 
as Figure \ref{Fig1}, except $Z$ and $Z'$ intersect in $q-1$ points. The cusps $[1]$ and $[\infty]$ 
again reduce to $E_1$ and $E_q$, respectively. (The reductions of $[T]$ and $[T^2]$ lie $Z$ and $Z'$.) 
One computes that $\Phi_\infty$ is generated by $z$, which has order $q^2$, and $e_1=-(q-1)z$. Now 
$$
\wp_\infty(c_1)=e_1-e_q=e_1-qe_1=(1-q)e_1=(q-1)^2z. 
$$
Since $q-1$ is coprime to $q$, we conclude that $\wp_\infty(c_1)$ 
generates $\Phi_\infty$. On the other hand, $c_1$ generates $\cC(\fn)$ and has order dividing $q^2$, so  
$$
\cC(\fn)\overset{\wp_\infty}{\cong}\Phi_\infty\cong \Z/q^2\Z. 
$$
 

\subsection{$\fn=T^2(T-1)$}\label{ssCDGT2} There are $6$ cusps given by the divisors of $T^2(T-1)$. 
By calculations 
similar to the previous case one shows that $\cC(\fn)$ is generated by $c_T$, and $c_T$ has order dividing $q(q^2-1)$. 

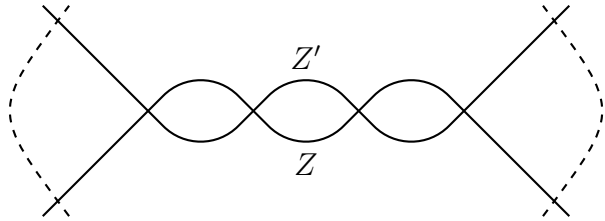
\begin{figure}[h]
\begin{tikzpicture}[scale=0.7, inner sep=.3mm]
\draw[thick, rounded corners=20pt]
(1,1) -- (4,4) -- (6,2) -- (8,4) -- (11,1);
\draw[thick, rounded corners=20pt]
(1,5) -- (4,2) -- (6,4) -- (8,2) -- (11,5);
\draw[thick, dashed] (1.5, 1) .. controls (0,3) .. (1.5, 5);
\draw[thick, dashed] (10.5, 1) .. controls (12,3) .. (10.5, 5);
\node at (6, 3.7) [label=above:$Z'$] {};
\node at (6, 2.3) [label=below:$Z$] {};
\end{tikzpicture}
\caption{$X_0(T^2(T-1))_{\F_\infty}$}\label{Fig3}
\end{figure}

The curve $X_0(\fn)$ has a regular model over $\cO_\infty$ whose special fibre is depicted in Figure \ref{Fig3}. 
Here $Z$ and $Z'$ are projective lines over $\F_\infty$ intersecting in $(q-2)$ points, and there are 
two chains of projective lines of length $q$ as in Figure \ref{Fig2}. Label the irreducible components so that 
$Z'$ intersects $E_1$. By \cite[Thm. 4.1]{PapikianJNT}, $\Phi_\infty\cong \Z/q(q^2-1)\Z$ is generated by $e_1:=E_1-Z'$. 
There is a cusp whose reduction lies on $E_1$ and there is a cusp whose reduction lies on $Z'$. In particular, $e_1\in \wp_\infty(\cC(\fn))$, 
so $\wp_\infty$ is surjective. Since the order of $\cC(\fn)$ is at most $q(q^2-1)$, we conclude 
$$
\cC(\fn)\overset{\wp_\infty}{\cong}\Phi_\infty\cong \Z/q(q^2-1)\Z. 
$$


\subsection{$\fn=T(T-1)(T-c), c\in \F_q\setminus \{0,1\}$}\label{ssCDGxyz} 
There are $8$ cusps indexed by the divisors of $\fn$. 
To simplify the notation, put $x:=T$, $y:=T-1$, $z:=T-c$. Using the strategy outlined earlier, one shows that the 
following relations hold for the $7$ generators of $\cC(\fn)$: 
$$
c_{yz}=c_1-c_x,\quad c_{xz}=c_1-c_y,\quad c_{xy}=c_1-c_z,  
$$
$$
c_z=(1-q)c_1+qc_x+qc_y, 
$$
$$
(q+1)c_x=(q+1)c_y=(q+1)c_z=(q-1)(q+1)c_1=0. 
$$
In particular, $\cC(\fn)$ is generated by $c_1$, $c_x$, $c_y$. 

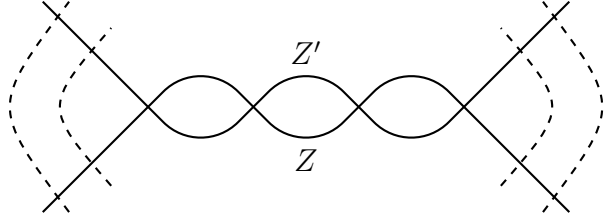
\begin{figure}[h]
\begin{tikzpicture}[scale=0.7, inner sep=.3mm]
\draw[thick, rounded corners=20pt]
(1,1) -- (4,4) -- (6,2) -- (8,4) -- (11,1);
\draw[thick, rounded corners=20pt]
(1,5) -- (4,2) -- (6,4) -- (8,2) -- (11,5);
\draw[thick, dashed] (1.5, 1) .. controls (0,3) .. (1.5, 5);
\draw[thick, dashed] (2.3, 1.5) .. controls (1,3) .. (2.3, 4.5);
\draw[thick, dashed] (10.5, 1) .. controls (12,3) .. (10.5, 5);
\draw[thick, dashed] (9.7, 1.5) .. controls (11,3) .. (9.7, 4.5);
\node at (6, 3.7) [label=above:$Z'$] {};
\node at (6, 2.3) [label=below:$Z$] {};
\end{tikzpicture}
\caption{$X_0(xyz)_{\F_\infty}$}\label{Fig4}
\end{figure}

$X_0(\fn)$ has a regular model over $\cO_\infty$ whose special fibre is depicted in Figure \ref{Fig4}. The two 
irreducible components $Z$ and $Z'$ are projective lines intersecting in $q-3$ points, and there are $4$ 
chains of projective lines of length $q$ as in Figure \ref{Fig2}. Label the projective lines in the first chain 
by $E_1, \dots, E_q$, the second by $F_1, \dots, F_q$, the third by $G_1, \dots, G_q$, and the fourth by $H_1, \dots, H_q$. 
Moreover, we can assume that $Z'$ intersects $E_1,F_1, G_1, H_1$. Similar to the previous cases, one computes that $\Phi_\infty$ 
is generated by $z:=Z-Z'$, $e_i:=E_i-Z'$, $f_i:=E_i-Z'$, $g_i:=E_i-Z'$, $h_i:=H_i-Z'$ ($1\leq i\leq q$) modulo 
the relations:
$$
e_i=i e_1, \quad f_i=i f_1, \quad g_i=i g_1, \quad h_i=i h_1 \qquad (1\leq i\leq q),
$$
$$
z=(q+1)e_1=(q+1)f_1=(q+1)g_1=(q+1)h_1,
$$
$$
(q+1)z=e_q+f_q+g_q+h_q,
$$
$$
(q-3)z+e_1+f_1+g_1+h_1=0.
$$
If we let $e:=e_1, f:=f_1, g:=g_1$, then these relations imply that 
$$
\Phi_\infty\cong \langle f-e\rangle\oplus \langle g-e\rangle\oplus \langle e\rangle\cong \Z/(q+1)\Z\oplus \Z/(q+1)\Z\oplus \Z/(q-1)^2(q+1)\Z. 
$$

The reductions of the cusps $[\infty], [x], [y], [z]$ lie on $E_1, F_1, G_1, H_1$, respectively, and the 
reductions of the cusps $[1], [yz], [xz], [xy]$ lie on $E_q, F_q, G_q, H_q$, respectively. Thus, 
$$
\wp_\infty(c_1)=E_q-E_1=e_q-e_1=(q-1)e,  
$$
$$
\wp_\infty(c_x)=F_1-E_1=f-e \quad \text{and}\quad \wp_\infty(c_y)=G_1-E_1=g-e. 
$$
This implies 
$$
\cC(\fn)\cong \langle c_x\rangle\oplus \langle c_y\rangle\oplus \langle c_1\rangle\cong \Z/(q+1)\Z\oplus \Z/(q+1)\Z\oplus \Z/(q^2-1)\Z
$$
and 
$$
\xymatrix{0 \ar[r] & \cC(\fn) \ar[r]^-{\wp_\infty}  & \Phi_\infty\ar[r] & \Z/(q-1)\Z \ar[r] & 0.}
$$

\begin{rem}
By \cite[Thm. 5.3]{PW}, $\Phi_x\cong \Phi_y\cong \Phi_z\cong \Z/(q+1)\Z$ and the canonical specialization maps 
from $\cC(xyz)$ into these component groups are surjective. 
\end{rem}


\section{Hecke algebras of small levels}\label{sHA}
We will need the descriptions of the graphs $\G_0(\fn)\bs \sT$ for $\deg(\fn)=3$. 
These graphs already appear in \cite{GekelerKleinem}. For the sake of completeness, and also because 
we will need explicit representatives for the edges $E(\G_0(\fn)\bs \sT)^+$, and need to know the orders 
of stabilizers of the edges (this was used in Section \ref{sCDG}), 
we describe these graphs below.  The graphs can be computed using the algorithm in \cite{GN}. 

\subsection{$\fn$ is irreducible}\label{ss31}
There are $q$ edges 
$$
b_u = \begin{pmatrix} \pi^3 & \pi + u \pi^2 \\0&1\end{pmatrix}, \quad u \in \F_q. 
$$

\begin{figure}[h]
\begin{tikzpicture}[scale=2, ->, >=stealth, semithick, inner sep=.5mm, vertex/.style={circle, fill=black}]

\node[vertex] (00) at (0, 0) {};
\node[vertex] (01) at (0, 1) {};
\node[vertex] (11) at (1, 1) {};
\node[vertex] (10) at (1, 0) {};
\node[vertex] (c1) at (-1, 0) {};
\node[vertex] (c2) at (2, 0) {};
\node at (.5, 1.05) {$\vdots\ b_u$};

\path[]
(11) edge[bend right]  (01)
(11) edge[bend left]  (01)
(01) edge node[auto,swap] {$a_\infty$} (00)
(10) edge node[auto] {$d_\infty$} (00)
(11) edge node[auto] {$a_1$} (10)
(00) edge[dashed] node[auto] {$s_\infty$} (c1)
(c2) edge[dashed] node[auto] {$s_1$} (10);
  
\end{tikzpicture}
\caption{$\G_0(\fn)\bs \sT$: $\fn$ irreducible}\label{Fig5}
\end{figure}

The dashed edges 
$$
s_\infty=\begin{pmatrix} \pi & 0 \\0&1\end{pmatrix}, \quad s_1=\begin{pmatrix} \pi^3 & 0 \\0&1\end{pmatrix}
$$
indicate that they are the first edges on a half-line corresponding to the cusps $[\infty]$ and $[1]$, respectively. Finally, 
$$
a_\infty=\begin{pmatrix} \pi^2 & \pi \\0&1\end{pmatrix}, \quad a_1=\begin{pmatrix} \pi^3 & \pi^2 \\0&1\end{pmatrix}, 
\quad 
d_\infty=\begin{pmatrix} \pi^2 & 0 \\0&1\end{pmatrix}. 
$$

A small calculation shows that 
\begin{equation}\label{eqweight}
w(a_\infty)=w(a_1)=q-1
\end{equation}
and the weights of all other edges in $(\G_0(\fn)\bs\sT)^0$ are $1$ (in particular, $w(\overline{a_\infty})=w(\overline{a_1})=1$). Also, 
the stabilizers in $\G_0(\fn)$ of preimages of all (non-oriented) edges in $(\G_0(\fn)\bs\sT)^0$ are isomorphic to $\F_q^\times$, except 
\begin{equation}\label{eqstab}
\#\Stab_{\G_0(\fn)}(d_\infty)=(q-1)^2. 
\end{equation}

\subsection{$\fn=T^3$} This case is similar to the case when $\fn$ is irreducible. We have 
$$
s_T=b_0=\begin{pmatrix} \pi^3 & \pi \\0&1\end{pmatrix}, \quad s_{T^2}=\begin{pmatrix} \pi^4 & \pi^2 \\0&1\end{pmatrix}. 
$$
The weights and stabilizers of edges in $(\G_0(\fn)\bs\sT)^0$ are the same as in (\ref{eqweight}) and (\ref{eqstab}). 
\begin{figure}[h]
\begin{tikzpicture}[scale=2, ->, >=stealth, semithick, inner sep=.5mm, vertex/.style={circle, fill=black}]

\node[vertex] (00) at (0, 0) {};
\node[vertex] (01) at (0, 1) {};
\node[vertex] (11) at (1, 1) {};
\node[vertex] (10) at (1, 0) {};
\node[vertex] (c1) at (-1, 0) {};
\node[vertex] (c2) at (2, 0) {};
\node[vertex] (c3) at (-1, 1) {};
\node[vertex] (c4) at (2, 1) {};
\node at (.5, 1.05) {$\vdots\ b_u$};

\path[]
(11) edge[bend right]  (01)
(11) edge[bend left]  (01)
(01) edge node[auto,swap] {$a_\infty$} (00)
(10) edge node[auto] {$d_\infty$} (00)
(11) edge node[auto] {$a_1$} (10)
(00) edge[dashed] node[auto] {$s_\infty$} (c1)
(c2) edge[dashed] node[auto] {$s_1$} (10)
(c3) edge[dashed] node[auto] {$s_T=b_0$} (01)
(c4) edge[dashed] node[auto, swap] {$s_{T^2}$} (11);
  
\end{tikzpicture}
\caption{$\G_0(T^3)\bs \sT$}\label{Fig6}
\end{figure}

\subsubsection{$\fn=T^2(T-1)$} In this case the quotient graph looks like Figure \ref{Fig7}.  
\begin{figure}[h]
\begin{tikzpicture}[scale=2, ->, >=stealth, semithick, inner sep=.5mm, vertex/.style={circle, fill=black}]

\node[vertex] (00) at (0, 0) {};
\node[vertex] (01) at (0, 1) {};
\node[vertex] (11) at (1, 1) {};
\node[vertex] (10) at (1, 0) {};
\node[vertex] (02) at (0,2) {};
\node[vertex] (12) at (1, 2) {};
\node at (.5, 1.05) {$\vdots\ b_u$};
\node[vertex] (c1) at (-1, 0) {};
\node[vertex] (c2) at (2, 0) {};
\node[vertex] (c3) at (-1, 1) {};
\node[vertex] (c4) at (2, 1) {};
\node[vertex] (c5) at (-1, 2) {};
\node[vertex] (c6) at (2, 2) {};

\path[]
(11) edge[bend right]  (01)
(11) edge[bend left]  (01)
(01) edge node[auto,swap] {$a_\infty$} (00)
(10) edge node[auto] {$d_\infty$} (00)
(11) edge node[auto] {$a_1$} (10)
(02) edge node[auto, swap] {$b_1$} (01)
(12) edge node[auto, swap] {$d_{T-1}$} (02)
(12) edge node[auto] {$a_{T^2}$} (11)
(00) edge[dashed] node[auto] {$s_\infty$} (c1)
(c2) edge[dashed] node[auto] {$s_1$} (10)
(c3) edge[dashed] node[auto] {$s_T=b_0$} (01)
(c4) edge[dashed] node[auto, swap] {$s_{T(T-1)}$} (11)
(c5) edge[dashed] node[auto] {$s_{T-1}$} (02)
(c6) edge[dashed] node[auto, swap] {$s_{T^2}$} (12)
;
  
\end{tikzpicture}
\caption{$\G_0(T^2(T-1))\bs \sT$}\label{Fig7}
\end{figure}
The edges $b_u$ in the middle of the figure are indexed by $u\in \F_q\setminus\{0,1\}$. In particular, there are no such edges 
when $q=2$. The representatives for the edges $b_0, b_1, b_u, d_\infty, s_\infty, s_1$ are the same as earlier. 
In addition to those, we have  
$$
s_{T-1}=\begin{pmatrix} \pi^4 & (T-1)^{-1} \\0&1\end{pmatrix}, \quad s_{T(T-1)}=\begin{pmatrix} \pi^4 & (T(T-1))^{-1} \\0&1\end{pmatrix}, 
\quad 
s_{T^2}=\begin{pmatrix} \pi^5 & \pi^2 \\0&1\end{pmatrix}, 
$$
$$
a_{T^2}=\begin{pmatrix} \pi^4 & \pi^2 \\0&1\end{pmatrix}, \quad d_{T-1}=\begin{pmatrix} \pi^4 & \pi+\pi^2 \\0&1\end{pmatrix}. 
$$
The weights of all edges in $(\G_0(\fn)\bs\sT)^0$ are $1$, except 
$$
w(a_\infty)=w(a_1)=w(\overline{b_1})=w(\overline{a_{T^2}})=q-1. 
$$
Similarly, $\#\Stab_{\G_0(\fn)}(e)/\F_q^\times =1$ for all (non-oriented) edges, except for $e=d_\infty, d_{T-1}$, when it is $q-1$. 
 
\subsection{$\fn=T(T-1)(T-c), c\in \F_q\setminus \{0,1\}$}\label{ssLastCase} 
In this case the quotient graph looks like Figure \ref{Fig8}.  Denote $x:=T$, $y:=T-1$, $z:=T-c$. 
\begin{figure}[h]
\begin{tikzpicture}[scale=2, >=stealth, semithick, inner sep=.5mm, vertex/.style={circle, fill=black}, 
decoration={markings, mark=at position 0.6 with {\arrow{>}}}
]

\node[vertex] (00) at (0, 0) {};
\node[vertex] (01) at (0, 1) {};
\node[vertex] (11) at (1, 1) {};
\node[vertex] (10) at (1, 0) {};
\node[vertex] (02) at (0,2) {};
\node[vertex] (12) at (1, 2) {};
\node[vertex] (a) at (-.5, 2.5) {};
\node[vertex] (b) at (-2.5, 3) {};
\node[vertex] (c) at (1.5, 2.5) {};
\node[vertex] (d) at (3.5, 3) {};
\node at (.5, 1.06) {$\vdots\ b_u$};
\node[vertex] (c1) at (-1, 0) {};
\node[vertex] (c2) at (2, 0) {};
\node[vertex] (c3) at (-1.3, 2.5) {};
\node[vertex] (c4) at (-3.3, 3) {};
\node[vertex] (c5) at (-1, 2) {};
\node[vertex] (c6) at (2, 2) {};
\node[vertex] (c7) at (2.3, 2.5) {};
\node[vertex] (c8) at (4.3, 3) {};

\draw[postaction={decorate}] (a)-- node[auto, swap] {$b_1$}(01);
\draw[postaction={decorate}] (b)-- node[auto, swap] {$b_c$}(01);
\draw[postaction={decorate}] (c)-- node[auto] {$a_y'$}(11);
\draw[postaction={decorate}] (d)--node[auto] {$a_z'$}(11);
\draw[postaction={decorate}] (01) -- node[auto,swap] {$a_\infty$} (00);
\draw[->, dashed](00)-- node[auto] {$s_\infty$} (c1);
\draw[->, dashed](c2)-- node[auto] {$s_1$} (10) ;
\draw[postaction={decorate}] (11).. controls (.5, 1.2).. (01);
\draw[postaction={decorate}] (11) .. controls (.5, .8).. (01);
\draw[->](10)-- node[auto] {$d_\infty$} (00);
\draw[postaction={decorate}] (11)-- node[auto] {$a_1$} (10);
\draw[postaction={decorate}] (02) -- node[auto] {$b_0$} (01);
\draw[postaction={decorate}] (12) -- node[auto,swap] {$d_x$} (02);
\draw[postaction={decorate}] (12) --node[auto,swap] {$a_x'$} (11);
\draw[postaction={decorate}] (c) -- node[auto,swap] {$d_y$}(a);
\draw[postaction={decorate}] (d) -- node[auto,swap] {$d_z$}(b);
\draw[->, dashed](c5) -- node[near start, above] {$s_x$} (02);
\draw[->, dashed] (c6) -- node[near start, above] {$s_{yz}$} (12);
\draw[->, dashed](c3) -- node[auto] {$s_y$} (a);
\draw[->, dashed](c4) -- node[auto] {$s_z$} (b);
\draw[->, dashed](c7) -- node[auto, swap] {$s_{xz}$} (c);
\draw[->, dashed](c8) -- node[auto, swap] {$s_{xy}$} (d);

\end{tikzpicture}
\caption{$\G_0(xyz)\bs \sT$}\label{Fig8}
\end{figure}
The edges $b_u$ in the middle of the figure are indexed by $u\in \F_q\setminus\{0,1,c\}$. In particular, there are no such edges 
if $q=3$. 
For the cusps we have
$$
s_{\infty}=\begin{pmatrix} \pi & 0 \\0&1\end{pmatrix}, \quad s_{1}=\begin{pmatrix} \pi^3 & 0 \\0&1\end{pmatrix}, 
$$
$$
s_{x}=\begin{pmatrix} \pi^4 & x^{-1} \\0&1\end{pmatrix}, \quad s_{y}=\begin{pmatrix} \pi^4 & y^{-1} \\0&1\end{pmatrix}, \quad 
s_{z}=\begin{pmatrix} \pi^4 & z^{-1} \\0&1\end{pmatrix}, 
$$
$$
s_{xy}=\begin{pmatrix} \pi^5 & (xy)^{-1} \\0&1\end{pmatrix}, \quad s_{yz}=\begin{pmatrix} \pi^5 & (yz)^{-1} \\0&1\end{pmatrix}, \quad 
s_{xz}=\begin{pmatrix} \pi^4 & (xz)^{-1} \\0&1\end{pmatrix}.  
$$
Next,
$$
a_x:=b_0=\begin{pmatrix} \pi^3 & x^{-1} \\0&1\end{pmatrix}, \quad a_y:=b_1=\begin{pmatrix} \pi^3 & y^{-1} \\0&1\end{pmatrix}, \quad 
a_z:=b_c=\begin{pmatrix} \pi^3 & z^{-1} \\0&1\end{pmatrix}, 
$$
$$
a_x'=\begin{pmatrix} \pi^4 & (yz)^{-1} \\0&1\end{pmatrix}, \quad a_y'=\begin{pmatrix} \pi^4 & (xz)^{-1} \\0&1\end{pmatrix}, \quad 
a_z'=\begin{pmatrix} \pi^4 & (xy)^{-1} \\0&1\end{pmatrix}, 
$$
$$
a_\infty=\begin{pmatrix} \pi^2 & \pi \\0&1\end{pmatrix}, \quad a_\infty':=a_1=\begin{pmatrix} \pi^3 & \pi^2 \\0&1\end{pmatrix},
$$
$$
b_u=\begin{pmatrix} \pi^3 & \pi+u\pi^2 \\0&1\end{pmatrix}, \quad u\in \F_q\setminus\{0,1,c\}. 
$$
Finally,
$$
d_x=\begin{pmatrix} \pi^4 & \pi+c\pi^3 \\0&1\end{pmatrix}, \quad 
d_y=\begin{pmatrix} \pi^4 & \pi+\pi^2+c\pi^3 \\0&1\end{pmatrix}, \quad 
d_z=\begin{pmatrix} \pi^4 & \pi+c\pi^2+c\pi^3 \\0&1\end{pmatrix},
$$
$$
d_\infty=\begin{pmatrix} \pi^2 & 0 \\0&1\end{pmatrix}. 
$$
The weights of edges $a_\infty,a_\infty', \overline{a_x},\overline{a_x'}, \overline{a_y}, \overline{a_y'},\overline{a_z}, 
\overline{a_z'}$ are $q-1$; 
the weights of all other edges in $(\G_0(\fn)\bs\sT)^0$ are $1$. We have 
$$
\#\Stab_{\G_0(\fn)}(e)/\F_q^\times=q-1, \quad \text{if }e=d_\infty, d_x,d_y,d_z,
$$ 
and $\#\Stab_{\G_0(\fn)}(e)/\F_q^\times=1$ for all other (non-oriented) edges of $(\G_0(\fn)\bs\sT)^0$. 

\subsection{The pairing}\label{sPairing}
Consider the bilinear $\T(\fn)$-equivariant pairing 
\begin{align}\label{GPairing}
\T(\fn)\times \cH_{0}(\fn, \Z) \to \Z\\
\nonumber (T, f) \mapsto (f|T)^\ast(1).
\end{align} 

\begin{thm}\label{thmPP}
When $\deg(\fn)=3$, the pairing (\ref{GPairing}) is perfect. 
\end{thm}

\begin{rem}
(1) In \cite{Analytical}, Gekeler proved that the pairing (\ref{GPairing}) is non-degenerate and becomes a perfect pairing after tensoring with $\Z[p^{-1}]$.
It is not known whether the pairing is perfect in general, without inverting $p$. \\
(2) In \cite{PW}, we already proved Theorem \ref{thmPP} and its corollaries for $\fn=T\fp$, where $\fp$ is prime of degree $2$. 
\end{rem}

\begin{proof}
First, observe that the map 
$f \mapsto (f(b_u))_{u\in \F_q}$
induces an isomorphism $\cH_0(\fn, \Z)\xrightarrow{\sim}\Z^{\oplus q}$ for $\fn$ irreducible or $\fn=T(T-1)(T-c)$. 
(It is clear from Figures \ref{Fig5} and \ref{Fig8} that any cycle in $\G_0(\fn)\bs\sT$ contains at least one of these edges.) 
Similarly, the map $f \mapsto (f(b_u))_{u\in \F_q^\times}$ 
induces an isomorphism $\cH_0(\fn, \Z)\xrightarrow{\sim}\Z^{\oplus {q-1}}$ for $\fn=T^3$ or $\fn=T^2(T-1)$. 
(In these cases, the edge $b_0$ lies on a cusp, so $f(b_0)=0$.) Hence the harmonic cochains $f_v\in \cH_0(\fn, \Z)$, defined 
by $f_v(b_u)=\delta_{v,u}$=(Kronecker's delta) form a $\Z$-basis of $\cH_0(\fn, \Z)$, where $v$ runs over $\F_q$ (resp. $\F_q^\times$) 
for $\fn$ square-free (resp. non-square-free). 

Let $\kappa\in \F_q$ and $f\in \cH_0(\fn, \Z)$.  We have 
$$
q(f|T_{T-\kappa})^\ast(1) =q^2f^\ast(T-\kappa)=
\sum_{w\in \pi\cO_\infty/\pi^{3}\cO_\infty}
f\left(\begin{pmatrix} \pi^3 & w\\ 0 & 1\end{pmatrix}\right)
\eta\left(-(T-\kappa)w\right) 
$$
\begin{align*}
= &f\left(\begin{pmatrix} \pi^3 & 0\\ 0 & 1\end{pmatrix}\right) +
\sum_{\beta\in \F_q^\times} f\left(\begin{pmatrix} \pi^3 & \beta \pi^2\\ 0 & 1\end{pmatrix}\right)
\eta\left(-(\pi^{-1}-\kappa)\beta \pi^2\right) \\
& +\sum_{u\in \F_q} \sum_{\beta\in \F_q^\times} 
f\left(\begin{pmatrix} \pi^3 & \beta(\pi+u \pi^2)\\ 0 & 1\end{pmatrix}\right)\eta\left(-(\pi^{-1}-\kappa)\beta(\pi+u \pi^2)\right). 
\end{align*}
Since the double class of $\begin{pmatrix} \pi^3 & w\\ 0 & 1\end{pmatrix}$ does not change if $w$ is replaced by 
$\beta w$ ($\beta\in \F_q^\times$), 
$f\left(\begin{pmatrix} \pi^3 & 0\\ 0 & 1\end{pmatrix}\right)=f(s_1)=0$, and 
$\sum_{\beta\in \F_q^\times}\eta(\beta\pi)=-1$, the above sum reduces to 
\begin{equation}\label{eq3.3}
 -f(a_1)+\sum_{u\in \F_q} f(b_u)(q\delta_{u, \kappa}-1). 
\end{equation}
The condition (ii$'$) from the definition of harmonic cochains implies   
$$
w(a_\infty)f(a_\infty)+w(d_\infty)f(d_\infty)=(q-1)f(a_\infty)+f(d_\infty)=0,
$$
$$
w(a_1)f(a_1)+w(\overline{d_\infty})f(\overline{d_\infty})=(q-1)f(a_1)+f(\overline{d_\infty})=0, 
$$
\begin{equation}\label{eq3.4}
f(a_\infty)=\sum_{u\in \F_q}f(b_u). 
\end{equation}
Therefore, $f(a_1)=-\sum_{u\in \F_q}f(b_u)$. Substituting this into (\ref{eq3.3}), we get 
\begin{equation}\label{eq-key}
(f|T_{T-u})^\ast(1)= f(b_u).
\end{equation}
In particular, $(f_v|T_{T-u})^\ast(1)=\delta_{u, v}$. This implies that the homomorphism 
\begin{equation}\label{eqTHZ}
\T(\fn)\to  \Hom(\cH_0(\fn, \Z), \Z)
\end{equation}
induced by the pairing (\ref{GPairing}) is surjective. Comparing the ranks of both sides, we conclude that this map is  
in fact an isomorphism, which is equivalent to the pairing being perfect. 
\end{proof}

Let $M$ be the $\Z$-submodule of $\T(\fn)$ generated by $\{T_{T-u}\ |\ u\in \F_q\}$ for $\fn$ square-free, and 
by $\{T_{T-u}\ |\ u\in \F_q^\times\}$ for $\fn$ non-square-free. From the previous calculations it is clear that 
the composition of $M\hookrightarrow \T(\fn)$ with (\ref{eqTHZ}) gives a surjection $M\to \Hom(\cH_0(\fn, \Z), \Z)$. 
This implies the following:

\begin{cor}\label{propTgen} Assume $\deg(\fn)=3$. If $\fn$ is square-free, then there is an isomorphism of $\Z$-modules 
$$\T(\fn)\cong \bigoplus_{u\in \F_q}\Z T_{T-u}.$$
If $\fn=T^3$ or $T^2(T-1)$, then 
$$\T(\fn)\cong \bigoplus_{u\in \F_q^\times}\Z T_{T-u}.$$
\end{cor}
 
Note that in our new notation, the equation (\ref{eqfast1}) is 
\begin{equation}\label{eqf(1)}
f^\ast(1)=-f(a_\infty).
\end{equation}
Denote $S=1+\sum_{u\in \F_q} T_{T-u}$. Using (\ref{eq-key}) we get 
$$
(f|S)^\ast(1) =f^\ast(1)+\sum_{u\in \F_q} f(b_u)=-f(a_\infty)+\sum_{u\in \F_q} f(b_u)=0,  
$$
where the last equality follows from (\ref{eq3.4}). 
The non-degeneracy of the pairing implies that $S=0$, i.e.,  
\begin{equation}\label{eqRelation1}
\sum_{u\in \F_q} T_{T-u}=-1. 
\end{equation}
On the other hand, if $T^2$ divides $\fn$, then $b_0$ lies on a cusp, so $(f|U_T)^\ast(1)=f(b_0)=0$. Thus, $U_T=0$. 
(Note that this also follows from Corollary \ref{cor1.8}.)
This implies that in Corollary \ref{propTgen} 
we can replace one of $T_{T-u}$ by $1$ and still have a $\Z$-basis of $\T(\fn)$. 

\begin{cor}\label{corTT0}
If $\deg(\fn)=3$ and $\fn\neq T(T-1)(T-c)$, then 
$$
\T(\fn)=\T(\fn)^0\cong \Z\oplus \bigoplus_{\substack{u\in \F_q \\ (T-u)\nmid \fn}}\Z T_{T-u}.
$$ 
\end{cor}

\begin{rem}
When $\fn = T(T-1)(T-c)$, we will show that $\T(\fn)/\fE(\fn)$ 
is not cyclic. As $\T(\fn)^0/\fE(\fn)^0$ is cyclic by Lemma~\ref{lemTE0}, this implies that $\T(\fn) \neq \T(\fn)^0$.  
\end{rem}


\section{The Eisenstein ideal of small levels} \label{sec4}

Our main goal in this section is to compute $\T(\fn)/\fE(\fn)$ when $\deg(\fn)=3$.  
We start with a few observations. First of all, by Corollary \ref{cor1.8}, the operator $U_\fp$ is 
either $0$ or $-W_\fp$. Since the Atkin-Lehner involutions commute 
with each other, this implies that $W_\fn U_\fp W_\fn =U_\fp$. Hence the sequence (\ref{eqmopa-inf}) 
is $\T(\fn)$-equivariant with respect to the usual action of $\T(\fn)$ on $\cH_0(\fn, \Z)$. Next, by Theorem \ref{thmPP},  
there is an isomorphism $\T(\fn)$-modules:
$$
\Hom(\cH_0(\fn, \Z), \Z)\cong \T(\fn). 
$$
Hence (\ref{eqmopa-inf}) gives a surjective $\T(\fn)$-equivariant homomorphism $\T(\fn)\to \Phi_\infty$. We will prove 
that this homomorphism factors through $\T(\fn)/\fE(\fn)$, and in fact this gives an isomorphism:  
\begin{thm}\label{thmTE}
If $\deg(\fn)=3$, then $\T(\fn)/\fE(\fn)\xrightarrow{\sim}\Phi_\infty$. 
\end{thm}

A crucial part of the proof consists of showing that $\Phi_\infty$ is annihilated by $\fE(\fn)$. Assume this fact for the moment. Then we get a surjection 
$
\T(\fn)/\fE(\fn)\twoheadrightarrow \Phi_\infty$. 
Now to prove Theorem \ref{thmTE},   
it is enough to show that the order of $\T(\fn)/\fE(\fn)$ is not larger than the order of $\Phi_\infty$. 
We will do this on a case-by-case basis. To simplify the notation, we omit $\fn$ and let $\T:=\T(\fn)$, $\fE:=\fE(\fn)$. 

\subsection{$\fn$ is irreducible} In this case, we know that $\T=\T^0$, so $\T/\fE\cong \Z/N\Z$ for some $N\geq 1$; see Corollary \ref{corTT0} 
and Lemma \ref{lemTE0}. By (\ref{eqRelation1}) 
$$
\sum_{u\in \F_q} T_{T-u}=-1. 
$$
Since $T_{T-u}\equiv (q+1)\ \mod\ \fE$, in $\T/\fE$ we have $q(q+1)+1=0$. 
Therefore $N$ divides $q^2+q+1$. On the other hand, from the calculations in Section \ref{sCDG} 
we know that $\Phi_\infty\cong \Z/(q^2+q+1)\Z$. 

\subsection{$\fn=T^3$} This case is very similar to the previous one. Again $\T/\fE\cong \Z/N\Z$ for some $N\geq 1$. 
The difference is that $U_T=0$ (see Corollary \ref{cor1.8}), so (\ref{eqRelation1})  implies that in $\T/\fE$
$$
(q-1)(q+1)+1=q^2=0.
$$ 
Hence $N$ divides $q^2=\# \Phi_\infty$.

\subsection{$\fn=T^2(T-1)$}\label{ss4.2.3} Again $\T/\fE\cong \Z/N\Z$ for some $N\geq 1$, but the above argument does not quite 
work since we only have 
$$
U_{T-1}+\sum_{\substack{u\in \F_q\\ u\neq 0,1}}T_{T-u}=-1. 
$$ 
(Note that $U_T=0$ by Corollary \ref{cor1.8}.) 
This implies that $U_{T-1}+(q^2-q-1)\in \fE$, but does not 
give a bound on the order of $1$ in $\T/\fE$. Instead we use a different approach. 

Any cuspidal harmonic cochain in $\cH_{00}(\fn, R)$ is uniquely determined by its values on the edges $\{b_u\ |\ u\in \F_q^\times\}$; 
cf. Figure \ref{Fig7}. Note that for any $u\in \F_q$ and $f\in \cH_{00}(\fn, R)$ the equations (\ref{eq-key}) and (\ref{eqf(1)}) 
imply 
\begin{equation}\label{eq-key2}
f(b_u)=-(f|T_{T-u})(a_\infty),
\end{equation}
as $f$ is in the image of $\cH_0(\fn, \Z)$. Now suppose $f\in \cE_{00}(\fn, R)$ and let $f(a_\infty)=\alpha$. 

If $q>2$, then equation (\ref{eq-key2}) gives 
$$
f(b_u)=-(q+1)\alpha, \quad u\in \F_q, u\neq 0,1.
$$
Next, the harmonicity implies
$$
f(b_1)=(q^2-q-1)\alpha, 
$$
so, in fact, $f$ is uniquely determined by $\alpha$. Let $g:=f|W_{T-1}$. We have 
$$
g=f|W_{T-1}=-f|U_{T-1}=f+\sum_{u\neq 0,1}f|T_{T-u}=((q+1)(q-2)+1)f=(q^2-q-1)f. 
$$
The action of $W_{T-1}$ on $\G_0(\fn)\bs\sT$ is easy to describe. Using Lemma \ref{lemWcusp}, we see that $W_{T-1}$ 
interchanges the cusps as follows: 
$$
[\infty] \longleftrightarrow [T^2], \quad [1] \longleftrightarrow [T-1], \quad [T] \longleftrightarrow [T(T-1)]. 
$$ 
Hence $W_{T-1}$ maps $b_u$, for any $u\neq 0,1$, to $\overline{b_{u'}}$ for some other $u'\neq 0,1$. Therefore, 
$$
g(b_u)=(q^2-q-1)f(b_u)=-(q^2-q-1)(q+1)\alpha
$$
and
$$
g(b_u)=f|W_{T-1}(b_u)=f(\overline{b_{u'}})=(q+1)\alpha.
$$
This implies $q(q^2-1)\alpha=0$. 

When $q=2$, let $\fp=T^2+T+1$. We have $f|T_\fp=(q^2+1)f=5f$. 
A direct calculation shows that 
$$
f|T_\fp(d_\infty)=f(d_{T-1}). 
$$
By harmonicity  
$$
w(a_\infty)f(a_\infty)+w(d_\infty)f(d_\infty)=(q-1)f(a_\infty)+f(d_\infty)=0,
$$
so $f(d_\infty)=-\alpha$. Similarly, $f(b_1)=-\alpha$ and $f(d_{T-1})=\alpha$. 
Thus 
$$
-5\alpha=f|T_\fp (d_\infty)=f(d_{T-1})=\alpha,
$$
which implies $6\alpha=q(q^2-1)\alpha=0$. 

The overall conclusion is that there is an injection 
\begin{align}\label{eqInjE}
\cE_{00}(\fn, R) &\hookrightarrow R[q(q^2-1)]\\ 
\nonumber f &\mapsto f(a_\infty).
\end{align}

Denote $\T_R:=\T\otimes_\Z R$ and let $\fE_R$ be the image of $\fE\otimes_\Z R$ in $\T_R$. 
By Theorem \ref{thmPP}, 
$$
\Hom_R (\T_R, R)\cong \cH_{00}(\fn, R). 
$$
Hence 
$$
\cE_{00}(\fn, R)\cong \Hom_R (\T_R, R)[\fE_R]\cong \Hom_R (\T_R/\fE_R, R). 
$$
(To see the second isomorphism note that an $R$-linear map $\psi: \T_R\to R$ is annihilated by $\fE_R$ 
if and only if $(e\circ\psi)(t)=\psi(et)=0$ for all $e\in \fE_R$ and all $t\in \T_R$. But since $\fE_R$ is an ideal in $\T_R$, this 
last condition is equivalent to $\psi$ vanishing on $\fE_R$, or in other words, $\psi$ must factor through $\T_R/\fE_R$.) 

Take $R=\Z/N\Z$. Then $\T_R/\fE_R\cong \T/\fE\cong \Z/N\Z$, so $\cE_{00}(\fn, \Z/N\Z)\cong \Z/N\Z$. 
On the other hand, the injection (\ref{eqInjE}) identifies $\cE_{00}(\fn, \Z/N\Z)$ with a subgroup of $\Z/N\Z[q(q^2-1)]$. 
Hence $N$ must divide $q(q^2-1)=\# \Phi_\infty$. 

\subsection{$\fn=T(T-1)(T-c)$, $c\in \F_q\setminus\{0,1\}$}
The argument here is similar to the previous case. With notation as in $\S$\ref{ssLastCase}, one checks that 
\begin{align*}
\cH_0(\fn, R)&\cong \bigoplus_{u\in \F_q} R \oplus R[q-1]\oplus R[q-1]\oplus R[q-1]\\
f &\mapsto (f(b_u), f(a_x)+f(a_x'), f(a_y)+f(a_y'), f(a_z)+f(a_z'))
\end{align*}
and 
\begin{align*}
\cH_{00}(\fn, R) & =\{f\in \cH_0(\fn, R)\ |\ f(a_x)+f(a_x')=f(a_y)+f(a_y')=f(a_z)+f(a_z')=0\}\\ &\cong \bigoplus_{u\in \F_q} R. 
\end{align*}

Let $f\in \cE_{00}(\fn, R)$ and denote
$$
\alpha=f(a_\infty), \quad \beta=f(a_x), \quad \gamma=f(a_y), \quad \delta=f(a_z).
$$ 
First assume $q>3$.  
Since $f$ is Eisenstein, $f(b_u)=-(q+1)\alpha$ for $u\in \F_q\bs \{0,1,c\}$. We have the relations (see (\ref{eqRelation1}) and Corollary \ref{cor1.8})
$$
U_x+U_y+U_z+\sum_{u\neq 0,1,c} T_{T-u}=-1
$$
and 
$$
W_x+U_x=W_y+U_y=W_z+U_z=0.
$$
Hence 
\begin{equation}\label{eq4.6'}
W_x+W_y+W_z= \sum_{u\neq 0,1,c}T_{T-u}+1,
\end{equation}
and 
\begin{equation}\label{eq4.6}
(f|W_x+W_y+W_z)=(f|\sum_{u\neq 0,1,c}T_{T-u}+1)=((q+1)(q-3)+1)f. 
\end{equation}
The action of the Atkin-Lehner involutions on $\G_0(\fn)\bs \sT$ is easy to deduce by analysing their action on the cusps. 
In particular, one easily checks that $W_x, W_y, W_z$ map any edge $b_u$ with $u\neq 0,1, c$ to $\overline{b_{u'}}$ for some $u'\neq 0,1,c$. 
Fix some $u_0\in \F_q\bs \{0,1,c\}$. On one hand, from what we just said, it follows that 
$$
(f|W_x+W_y+W_z)(b_{u_0})=f(\overline{b_{u'}})+f(\overline{b_{u''}})+f(\overline{b_{u'''}})=3(q+1)\alpha. 
$$
On the other hand, from (\ref{eq4.6})
$$
(f|W_x+W_y+W_z)(b_{u_0})=(q^2-2q-2)f(b_{u_0})=-(q+1)(q^2-2q-2)\alpha. 
$$
This implies $(q+1)(q-1)^2\alpha=0$, i.e., $\alpha\in R[(q+1)(q-1)^2]$. 

Note that $f|W_x$, $f|W_y$, $f|W_z$ are in $\cE_{00}(\fn, R)$, 
since the Atkin-Lehner involutions commute with $T_\fp$ ($\fp\nmid \fn$). Fix some $u\neq 0,1,c$. Then 
$$
(f|W_x|T_{T-u})(a_\infty)=(q+1)(f|W_x)(a_\infty)=(q+1)f(\overline{a_{x'}})=(q+1)f(a_x)=(q+1)\beta. 
$$
Computing the same expression differently, 
$$
(f|W_x|T_{T-u})(a_\infty)=-(f|W_x)(b_u)=-f(\overline{b_{u'}})=f(b_{u'})=-(q+1)\alpha. 
$$
(Here we have used (\ref{eq-key2}).) Hence $(q+1)(\alpha+\beta)=0$, i.e., $\alpha+\beta\in R[q+1]$. 
Similarly, $\alpha+\gamma\in R[q+1]$. Finally, from (\ref{eq4.6}), we get 
$$
\beta+\gamma+\delta=(q^2-2q-2)\alpha,
$$
which means that $\delta$ is determined by $\alpha, \beta, \gamma$. 

Now assume $q=3$. In this case the previous argument needs to be modified as there are no edges $b_u$ with $u\neq 0,1,c$. 
Here (\ref{eq4.6'}) becomes 
$$
W_x+W_y+W_z=1. 
$$
Multiplying this expression by $W_x$, $W_y$ and $W_z$, and then adding the resulting relations, we get 
\begin{equation}\label{eq3Ws}
W_xW_y+W_xW_z+W_yW_z=W_{xy}+W_{xz}+W_{yz}=W_{xyz}=-1. 
\end{equation}
As before, let $f\in \cE_{00}(\fn, R)$. Note that 
$$
\beta=f(a_x)=(f|W_{yz})(\overline{a_\infty})=-(f|W_{yz})(a_\infty),
$$
and similarly 
$$
\gamma=-(f|W_{xz})(a_\infty), \quad \delta=-(f|W_{xy})(a_\infty). 
$$
Let $\fp_1=T^2+1$, $\fp_2=T^2+T-1$, $\fp_3=T^2-T-1$. By an explicit calculation 
\begin{align}\label{eqseveral}
(f|T_{\fp_1})(a_\infty)&=2f(a_x)-2f(a_y)-2f(a_z),\\ 
\nonumber (f|T_{\fp_2})(a_\infty)&=-2f(a_x)+2f(a_y)-2f(a_z),\\ 
\nonumber (f|T_{\fp_3})(a_\infty)&=-2f(a_x)-2f(a_y)+2f(a_z).  
\end{align}
On the other hand, $f|T_{\fp_i}=(q^2+1)f=10 f$. This implies 
\begin{equation}\label{eq3a}
4f(a_x)=4f(a_y)=4f(a_z). 
\end{equation}
By (\ref{eq3Ws}),
\begin{align*}
-4f(a_\infty) &=4f|(W_{xy}+W_{xz}+W_{yz})(a_\infty)\\
 &=4(f|W_x)(a_y)+4(f|W_y)(a_z)+4(f|W_z)(a_x). 
\end{align*}
Since $f|W_x, f|W_y, f|W_z\in \cE_{00}(\fn, R)$, we can apply (\ref{eq3a}) to each of these functions individually 
to conclude that the previous expression is equal to 
$$
4(f|W_x)(a_x)+4(f|W_y)(a_y)+4(f|W_z)(a_z)  
$$
$$
=4f(a_\infty)+ 4f(a_\infty)+4f(a_\infty)=12f(a_\infty). 
$$
Hence $16 f(a_\infty)=0$, i.e., $f(a_\infty)\in R[(q+1)(q-1)^2]$. 
Finally, multiply (\ref{eqseveral}) by $2$ and use $16 f(a_\infty)=0$ to get 
$$
4f(a_\infty)=4f(a_x)-4f(a_y)-4(a_z). 
$$
Using (\ref{eq3a}), we get 
$$
4(f(a_x)+f(a_\infty))=0. 
$$
By a similar argument, $4(f(a_y)+f(a_\infty))=4(f(a_z)+f(a_\infty))=0$. 

The overall conclusion is that for any $q\geq 3$, there is an injection 
\begin{align*}
\cE_{00}(\fn, R) &\hookrightarrow R[(q+1)(q-1)^2]\oplus R[q+1]\oplus R[q+1]\\ 
f &\mapsto (f(a_\infty), f(a_x)+f(a_\infty), f(a_y)+f(a_\infty)). 
\end{align*}

Let $N$ be the exponent of $\T/\fE$ as an abelian group, and $R=\Z/N\Z$. Then $\T_R/\fE_R\cong \T/\fE$, and 
$$
\cE_{00}(\fn, R)\cong \Hom_R(\T_R/\fE_R, R)\cong \T/\fE. 
$$
This implies that $\T/\fE$ is a subgroup of $\Z/N\Z[(q+1)(q-1)^2]\oplus \Z/N\Z[q+1]\oplus \Z/N\Z[q+1]$. 
This latter group is obviously a subgroup of $\Z/(q+1)(q-1)^2\Z\oplus \Z/(q+1)\Z\oplus \Z/(q+1)\Z\cong \Phi_\infty$. 


\subsection{$\Phi_\infty$ is Eisenstein} Here we prove the fact that was needed in the proof of Theorem \ref{thmTE}: 
\begin{prop}\label{propPhiEis} When $\deg(\fn)=3$, 
the component group $\Phi_\infty$ is Eisenstein.
\end{prop} 

When $\fn=T^3$ or $T^2(T-1)$, the proposition easily follows from our earlier results. Indeed, we have shown that in these cases 
the canonical specialization map $\wp_\infty: \cC(\fn)\to \Phi_\infty$ is an isomorphism. Since $\wp_\infty$ 
is $\T(\fn)$-equivariant and $\cC(\fn)$ is annihilated by $\fE(\fn)$ (see Lemma \ref{lem2.4}), the claim follows. 
When $\fn$ irreducible or $\fn=T(T-1)(T-c)$, we will need a different argument. 

\subsubsection{$\fn$ is irreducible} 
There is an isomorphism (cf. Figure \ref{Fig5})
\begin{align*}
\cH_{0}(\fn, R) &\xrightarrow{\sim} \bigoplus_{u\in \F_q} R\\ 
f &\mapsto (f(b_u)). 
\end{align*}
Let $h_u\in \cH_0(\fn, \Z)$ be the harmonic cochain 
defined by $h_u(b_{u'})=\delta_{u, u'}$. Then the set $\{h_u\}$ forms a $\Z$-basis of $\cH_0(\fn, \Z)$. 
Enumerate the elements of $\F_q$ from $0$ to $q-1$. Let $f_0=h_0$ and $f_i=h_i-h_{i-1}$, $1\leq i\leq q-1$. 
This is again a basis of $\cH_0(\fn, \Z)$. Let $\{\psi_i\}$ be the dual basis of $\Hom(\cH_0(\fn, \Z), \Z)$, i.e., 
$\psi_i(f_j)=\delta_{ij}$. 

It is easy to see that $f_0(a_\infty)=1$, $f_0(d_\infty)=-(q-1)$, $f_0(a_1)=-1$, and $f_0(b_u)=\delta_{0,u}$. 
Similarly, $f_i(b_j)=\delta_{i,j}-\delta_{i-1, j}$ for $1\leq i\leq q-1$. Recall that $n(d_\infty)=q-1$, and $n(e)=1$ 
for all other edges of $(\G_0(\fn)\bs\sT)^0$; for the definition of $n(e)$ see (\ref{eqn(e)}). Now one easily computes that the pairing (\ref{eqPIP}) gives 
$$
(f_0, f_0)=q+2, \quad (f_0, f_1)=-1,\quad  (f_0, f_i)=0 \quad \text{for }i\geq 2; 
$$
and for $1\leq i\leq q-1$
$$
(f_i,f_j) = \begin{cases} 2, & \text{if $i =j$,}\\ 
-1, & \text{if $j=i-1$ or $j=i+1$,} \\ 
0, & \text{otherwise.}
\end{cases}
$$
Hence the map $\iota:f\mapsto (f, \cdot)$ in (\ref{eqmopa-inf}) sends 
\begin{align*}
f_0 &\mapsto (q+2)\psi_0-\psi_1\\
f_i &\mapsto -\psi_{i-1}+2\psi_i-\psi_{i+1}\quad \text{for }1\leq i\leq q-2\\
f_{q-1} &\mapsto -\psi_{q-2} +2\psi_{q-1}. 
\end{align*}
Applying the exact sequence (\ref{eqmopa-inf}), we conclude that $\Phi_\infty$ 
is generated by $\psi_0$ modulo the relation $(q^2+q+1)\psi_0=0$, i.e., $\Phi_\infty\cong \Z/(q^2+q+1)\Z$. 

\begin{rem}
Of course, we already knew that $\Phi_\infty\cong \Z/(q^2+q+1)\Z$ from the computations in Section \ref{sCDG}. 
The advantage of using (\ref{eqmopa-inf}) is that it relates $\Phi_\infty$ 
to the cuspidal harmonic cochains, and so allows to compute the action of Hecke operators on $\Phi_\infty$.  
On the other hand, the method used in Section \ref{sCDG} is better suited for computing the specialization map
$\cC(\fn)\to \Phi_\infty$. 
\end{rem}

Let $N=q^2+q+1$. Consider the commutative diagram obtained from (\ref{eqmopa-inf}) by multiplication by 
$N$: 
\begin{equation}\label{eq-comdia}
\xymatrix{
0\ar[r] & \cH_0(\fn, \Z) \ar[r] \ar[d]^-{N} &  \Hom(\cH_0(\fn, \Z), \Z) \ar[r]\ar[d]^-{N} & \Phi_\infty\ar[r]\ar[d]^-{N} & 0 \\
0\ar[r] & \cH_0(\fn, \Z) \ar[r] &  \Hom(\cH_0(\fn, \Z), \Z) \ar[r] & \Phi_\infty\ar[r] & 0
}
\end{equation}
The snake lemma gives an injection $\Phi_\infty\hookrightarrow \cH_{00}(\fn, \Z/N\Z)$. Our previous calculations allow us
to explicitly describe the image of $\Phi_\infty$. Indeed, in $\Hom(\cH_0(\fn, \Z), \Z)$ we have 
$$
N\psi_0=\sum_{i=0}^{q-1}(q-i)\cdot \iota(f_i).  
$$
Therefore, the image of $\Phi_\infty$ in $\cH_{00}(\fn, \Z/N\Z)$ is generated by the cuspidal harmonic cochain 
$$
f:=qf_0+(q-1)f_1+\cdots+f_{q-1}\quad \mod\ N. 
$$
Since the above commutative diagram is compatible with the action of $\T(\fn)$, to conclude that $\Phi_\infty$ 
is Eisenstein, now it is enough to show that $f\in \cE_{00}(\fn, \Z/N\Z)$. 

\begin{defn}
Assume for the moment that $\fn\lhd A$ is an arbitrary non-zero ideal. Define 
$$
\nu(\fn) =
\begin{cases}
1, & \text{if  $\deg(\fn)$ is even};\\ 
q+1, & \text{if  $\deg(\fn)$ is odd.} 
\end{cases}
$$ 
The \textit{Eisenstein series} of level $\fn$ is the function on $E(\sT)$ defined in terms of Fourier expansion 
$$
E_{\fn}\left(\begin{pmatrix} \pi^k & y \\0 & 1\end{pmatrix}\right) = \nu(\fn)\cdot q^{-k+1} \cdot \left[
\frac{1-|\fn|}{1-q^2} + \sum_{0 \neq m \in A, \atop \deg(m) \leq k-2} \sigma_{\fn}(m) \eta(my)\right],
$$
where $\sigma_{\fn}(m):= \sigma(m) - |\fn| \cdot \sigma(m/\fn)$, and $\sigma$ is the divisor function
$$\sigma(m):= \begin{cases}\sum\limits_{\text{monic } m' \in A, \atop m' \mid m} |m'|, &\text{ if $m \in A$;} \\
0, & \text{ otherwise.}\end{cases}
$$
(Here $\eta$ is the additive character in (\ref{eq-eta}).) This function is in fact a (non-cuspidal) harmonic cochain in $\cH(\fn, \Z)$; 
see \cite{Discriminant}. Note that for each prime $\fp \lhd A$ and $m \in A$, 
$$\sigma(m\fp) = \sigma(\fp) \sigma(m) - |\fp|\sigma(m/\fp).$$
Therefore the Fourier expansion of $E_{\fn}|T_{\fp}$ for each prime $\fp$ not dividing $\fn$ tells us that 
\begin{equation}\label{eqEisisEis}
E_{\fn}|T_{\fp} = (|\fp|+1)E_{\fn}.
\end{equation}
\end{defn}

Returning to the case when $\fn$ is irreducible of degree $3$, we relate $f$ to the reduction of $E_\fn$ modulo $N$. 
From the definition of $f$, it is easy to see that $f(b_u)=1$ for all $u\in \F_q$, and as we have discussed, this uniquely 
characterizes $f$ in $\cH_{00}(\fn, \Z/N\Z)$. Next, we compute 
$$
E_{\fn}(s_\infty)=E_{\fn}\left(\begin{pmatrix} \pi & 0 \\0 & 1\end{pmatrix}\right)=(q+1)\frac{q^3-1}{q^2-1}=q^2+q+1. 
$$
$$
E_{\fn}(b_u)=E_{\fn}\left(\begin{pmatrix} \pi^3 & \pi+u\pi^2 \\0 & 1\end{pmatrix}\right)=\frac{(q+1)}{q^2}
\left[
\frac{q^3-1}{q^2-1} + \sum_{0 \neq m \in A, \atop \deg(m) \leq 1} \sigma_{\fn}(m) \eta(m(\pi+u\pi^2))\right]
$$
$$
=\frac{(q+1)}{q^2}
\left[
\frac{q^3-1}{q^2-1} + \sum_{m\in \F_q^\times} \eta(m(\pi+u\pi^2))+q\sum_{\deg(m)=1} \eta(m(\pi+u\pi^2))\right]
$$
$$
=\frac{(q+1)}{q^2}
\left[
\frac{q^3-1}{q^2-1} -1+0\cdot q\right]=1. 
$$
Therefore,
$$
E_\fn\equiv f\ \mod\ N, 
$$
and (\ref{eqEisisEis}) implies $f|T_\fp=(|\fp|+1)f$ as was required to show. 

\subsubsection{$\fn=T(T-1)(T-c)$, $c\in \F_q\setminus \{0,1\}$} The strategy of the proof of Proposition \ref{propPhiEis} in this case 
is similar to the case when $\fn$ is irreducible, but the calculations become much more complicated. 
As before, any harmonic cochain in $\cH_0(\fn, R)$ is uniquely determined by its values on the edges $\{b_u\ |\ u\in \F_q\}$; 
see Figure \ref{Fig8}.  
Let $h_u\in \cH_0(\fn, \Z)$ be defined by $h_u(b_{u'})=\delta_{u, u'}$. Then $\{h_u\ |\ u\in \F_q\}$ is a 
$\Z$-basis of $\cH_0(\fn, \Z)$. Let 
\begin{align*}
f_1 &=-h_0-h_1+(q^2-2q)h_c-\sum_{u\neq 0,1,c}h_u,\\ 
f_2 &=h_0-h_c,\\ 
f_3 &= h_1-h_c. 
\end{align*}
The pairing (\ref{eqPIP}) gives
$$
(h_0, f_1)=0, \quad (h_1, f_1)=0, \quad (h_c, f_1)=(q+1)(q-1)^2, \quad (h_u, f_1)=0 \text{ for }u\neq 0,1,c;
$$
$$
(h_0, f_2)=q+1, \quad (h_1, f_2)=0, \quad (h_c, f_2)=-(q+1), \quad (h_u, f_2)=0 \text{ for }u\neq 0,1,c;
$$
$$
(h_0, f_3)=0, \quad (h_1, f_3)=q+1, \quad (h_c, f_3)=-(q+1), \quad (h_u, f_3)=0 \text{ for }u\neq 0,1,c.
$$
This implies that $\psi_1:=\frac{\iota(f_1)}{(q+1)(q-1)^2}$, $\psi_2:=\frac{\iota(f_2)}{(q+1)}$, 
$\psi_3:=\frac{\iota(f_3)}{(q+1)}$ are in $\Hom(\cH_0(\fn, \Z), \Z)$. One computes that 
$$
\Phi_\infty\cong \frac{\Hom(\cH_0(\fn, \Z), \Z)}{\iota(\cH_0(\fn, \Z))}\cong \frac{\Z}{(q+1)(q-1)^2\Z}\psi_1
\oplus \frac{\Z}{(q+1)\Z}\psi_2\oplus \frac{\Z}{(q+1)\Z}\psi_3. 
$$
Using the diagram (\ref{eq-comdia}) with $N=(q+1)(q-1)^2$, we get an injection from $\Phi_\infty$ into $\cH_{00}(\fn, \Z/N\Z)$, 
and the image is generated by $f_1$, $f_2$, and $f_3$. Now it is enough to show that the reductions of these harmonic cochains are Eisenstein. 

Let $x:=T, y:=T-1, z:=T-c$. The Eisenstein series $E_x, E_y, E_z, E_{xy}, E_{xz}, E_{yz}, E_{xyz}$ are in $\cH(\fn, \Z)$. 
Further, denote
$$
E_{xy}'=E_x-E_y, \quad E_{yz}'=E_y-E_z, \quad E_{xyz}'=E_{xy}+(q+1)^{-1}(E_{xyz}-E_z). 
$$
One can check that $E_{xyz}'$ is $\Z$-valued by computing its values on the edges of $\G_0(\fn)\bs \sT$. (Note that the  
matrix representatives of the edges are given after Figure \ref{Fig8}.) Moreover, by computing the values of other Eisenstein series 
on the edges of $\G_0(\fn)\bs \sT$ (see Table 1), it is possible to show that modulo $N$ we have
$$
f_1 \equiv -(q+1)E_x+(q^2+1)E_{xy}+(q^2+q+1)E_{xy}'+qE_{yz}+(q^2+1)E_{yz}'+qE_{xz}-qE_{xyz}',
$$
and modulo $q+1$ we have 
\begin{align*}
f_2 &\equiv -E_{xy}-E_{xy}'+E_{yz}-E_{yz}'\\
f_3 &\equiv -E_{xy}-E_{yz}'+E_{xz}. 
\end{align*}
This proves that $\Phi_\infty$ is Eisenstein. 

\begin{table}[h]\label{tableEis}
\caption{Values of Eisenstein series on $\G_0(xyz)\bs \sT$}
\scalebox{0.7}{
\begin{tabular}{ |c || c | c | c | c |c |c |c |c |c |c |c |c |c |c |c |c |c |c |c |c |c |}
\hline
 &  $s_\infty$ & $s_1$& $s_x$& $s_y$& $s_z$& $s_{yz}$& $s_{xz}$& $s_{xy}$& $d_\infty$& $d_x$& $d_y$& $d_z$& $a_\infty$& $a_\infty'$& $a_x$& $a_x'$& $a_y$& $a_y'$& $a_z$& $a_z'$& $b_u$\\
\hline
$E_x$ &  $1$ & $q^2$& $-q^2$& $q$& $q$& $1$& $-q$& $-q$& $q$& $q$& $-1$& $-1$& $-1$& $-q$& $-q$& $-1$& $1$& $-1$& $1$& $-1$& $1$\\
\hline
$E_{xy}$ & $1$& $q$& $0$& $0$& $1$& $0$& $0$& $-q$& $1$& $0$& $0$& $-1$& $0$& $-1$& $0$& $0$& $0$& $0$& $0$& $-1$& $0$ \\
\hline
$E_{xy}'$ & $0$& $0$& $-q$& $q$& $0$& $1$& $-1$& $0$& $0$& $1$& $-1$& $0$& $0$ & $0$& $-1$& $0$& $1$& $0$& $0$& $0$& $0$\\
\hline
$E_{yz}$ & $1$& $q$& $1$& $0$& $0$& $-q$& $0$& $0$& $1$& $-1$& $0$& $0$& $0$& $-1$& $0$& $-1$& $0$& $0$& $0$& $0$& $0$ \\
\hline
$E_{yz}'$ & $0$& $0$& $0$& $-q$& $q$& $0$& $1$& $-1$& $0$& $0$& $1$& $-1$& $0$& $0$& $0$& $0$& $-1$& $0$& $1$& $0$& $0$\\
\hline
$E_{xz}$ & $1$& $q$& $0$& $1$& $0$& $0$& $-q$& $0$& $1$& $0$& $-1$& $0$& $0$& $-1$& $0$& $0$& $0$& $-1$& $0$& $0$& $0$ \\
\hline
$E_{xyz}'$ & $q + 1$& $q + 1$& $0$& $0$& $q + 1$& $0$& $0$& $-q - 1$& $2$ & $0$& $0$& $-2$& $1$& $-1$& $0$& $0$& $0$& $0$& $1$& $-1$& $0$\\ 
\hline
\end{tabular}
}
\end{table}


\subsection{Alternative definitions of the Eisenstein ideal}\label{ssAltDef} The original definition of the Eisenstein ideal in \cite{Mazur} 
includes also a congruence on the $U_p$ operator. 
Another definition of Eisenstein ideal 
for prime level $\fp$ is given in \cite[p. 230]{Tamagawa} as the kernel of $\T(\fp)\to \End_\Z(\cC(\fp))$. 

For non-zero $\fn\lhd A$, the cuspidal divisor group $\cC(\fn)$ is preserved by the action of $\T(\fn)$, so one can consider   
$$
\fI(\fn)=\ker(\T(\fn)\to \End_\Z(\cC(\fn))). 
$$
Since, in general, $\cC(\fn)$ is not Eisenstein, $\fI(\fn)$ will be different from the Eisenstein ideal. On the other hand, when 
$\deg(\fn)=3$, $\cC(\fn)$ is Eisenstein (see Lemma \ref{lem2.4}), so $\fE(\fn)\subseteq \fI(\fn)$. 

\begin{lem}\label{lem417} Assume $\deg(\fn)=3$. 
\begin{itemize}
\item[(i)] If $\fn$ is irreducible, or $\fn=T^3$, or $\fn=T^2(T-1)$, then $\fE(\fn)=\fI(\fn)$. 
\item[(ii)] If $\fn=xy$, where $x$ and $y$ are irreducible of degree $1$ and $2$, respectively, then 
$\fE(\fn)= \fI(\fn)$ if and only if $q$ is even. 
\end{itemize}
\end{lem}
\begin{proof} For all these $\fn$ we know that $\T(\fn)/\fE(\fn)\cong \Z/N(\fn)\Z$, where $N(\fn)$ 
is the order of $\cC(\fn)$. 
The inclusion $\fE(\fn)\subseteq \fI(\fn)$ induces a surjection $\T(\fn)/\fE(\fn)\to \T(\fn)/\fI(\fn)$, 
so this latter group is also cyclic: $\T(\fn)/\fI(\fn)\cong \Z/M(\fn)\Z$, where $M(\fn)|N(\fn)$. 
From the definition of $\fI(\fn)$ we see that $M(\fn)$ is the exponent of the group $\cC(\fn)$. 
If $\fn$ is irreducible, or $\fn=T^3$, or $\fn=T^2(T-1)$, $\cC(\fn)$ is a cyclic, so $M(\fn)=N(\fn)$. For $\fn=xy$, 
 $\cC(\fn)$ is cyclic if and only if $q$ is even. 
\end{proof}

If $\fn$ is one of the ideals in the previous lemma, then $\T(\fn)/\fE(\fn)$ is cyclic. This implies that $U_\fp$ 
is congruent to an integer modulo $\fE(\fn)$. If $\fp^2|\fn$, then we know that $U_\fp=0$. If 
$\fp\parallel \fn$, then $U_\fp=-W_\fp$, so one can determine the integer in question 
by studying the action of $W_\fp$ on $\Phi_\infty$ (since $\Phi_\infty\cong \T(\fn)/\fE(\fn)$). This is easy to do by considering 
the action of $W_\fp$ on the cusps. The results are the following:
\begin{enumerate}
\item If $\fn$ is irreducible, then $U_\fn-1\in \fE(\fn)$. 
\item If $\fn=T^2(T-1)$, then $U_{T-1}+(q^2-q-1)\in \fE(\fn)$. 
\item If $\fn=xy$, where $x$ and $y$ are irreducible of degree $1$ and $2$, respectively, then 
$U_x+q^2$ and $U_y-q^2$ are in $\fE(\fn)$. 
\item If $\fn=xyz$, where $x$, $y$, and $z$ are distinct irreducibles of degree $1$, then by a similar calculation 
one can show that none of $U_x, U_y, U_z$ is a scalar modulo $\fE(\fn)$, although 
\begin{align*}
U_xU_yU_z=-W_\fn &\equiv 1\ (\mod\ \fE(\fn)),\\ 
U_x+U_y+U_z &\equiv -(q^2-2q-2)\ (\mod\ \fE(\fn)).
\end{align*}
\end{enumerate}

\begin{lem}
Let $\fp\lhd A$ be an arbitrary prime. Then $\fE(\fp)=\fI(\fp)$ and $U_\fp-1\in \fE(\fp)$. 
\end{lem}
\begin{proof} 
In the notation of Section \ref{sEPN}, 
the cuspidal divisor group $\cC(\fp)$ is generated by $c_0=[0]-[\infty]$, which has order $N(\fp)$ given in (\ref{eqNfp}). 
By Theorem \ref{thmEispnot}, $\T(\fp)/\fE(\fp)\cong \Z/N(\fp)\Z$. Since $c_0$ is Eisenstein, 
we can repeat the argument in Lemma \ref{lem417} to get the equality $\fE(\fp)=\fI(\fp)$. Since $W_\fp=-U_\fp$ 
acts as $-1$ on $c_0$, we have $U_\fp-1\in \fI(\fp)$. 
\end{proof}

\begin{rem}
The previous lemma implies that we could have defined $\fE(\fp)$ to be the ideal generated by $T_\fq-(|\fq|+1)$ for all $\fq\neq \fp$, and $U_\fp-1$. 
This is the exact analogue of Mazur's definition of the Eisenstein ideal in \cite{Mazur}. 
\end{rem}


\section{Integral version of a theorem of Atkin and Lehner}\label{sIAL} Let $\fn=T(T-1)(T-c)$, where $c\in \F_q\bs\{0,1\}$. 
We proved that $\T(\fn)/\fE(\fn)$ is not a cyclic group, for example, it contains three distinct subgroups of order $q+1$. 
On the other hand, by Lemma \ref{lemTE0}, $\T(\fn)^0/\fE(\fn)^0$ is cyclic. Hence the index $[\T(\fn):\T(\fn)^0]$ 
is finite but strictly larger than $1$. 
In this subsection we deduce some information about this index as a consequence of a general result which 
can be considered as a certain integral version of Theorem 1 in \cite{AL}.  

Following \cite{Pal}, we say that $R$ is a \textit{coefficient ring} if $p\in R^\times$
and $R$ is a quotient of a discrete valuation ring which contains $p$-th roots of unity. 
As is observed in \cite{Pal}, the theory of Fourier expansions discussed in $\S$\ref{ssFE} works 
over any coefficient ring. 

Let $\fm\in A$ be a non-zero ideal and denote
$$
B_\fm=\begin{pmatrix} \fm & 0 \\  0 & 1\end{pmatrix}. 
$$
It is easy to check that for any $f\in \cH(\fn, R)$,  we have $f|B_\fm\in \cH(\fn\fm, R)$.

\begin{thm}\label{thmIntAL}
Let $\fn\lhd A$ be a non-zero ideal. Suppose there are three distinct prime ideals $\fp_1, \fp_2, \fp_3$ of $A$ which divide 
$\fn$ but are coprime to $\fn/(\fp_1\fp_2\fp_3)$. 
Let $f\in \cH(\fn, R)$, where $R$ is a coefficient ring. 
Suppose $f^\ast(\fm)=0$ unless $\fp_i|\fm$ for some $i=1,2,3$. 
Then there exist $f_i\in \cH(\fn/\fp_i, R)$, $1\leq i\leq 3$, such that 
$$
s_{\fp_1, \fp_2}s_{\fp_1, \fp_3}s_{\fp_2, \fp_3}\cdot f=f_1|B_{\fp_1}+f_2|B_{\fp_2}+f_3|B_{\fp_3}, 
$$
where $s_{\fp_i, \fp_j}=\mathrm{gcd}(|\fp_i|+1, |\fp_j|+1)$. 
\end{thm}
\begin{proof} 
Let $$\phi_3:=|\fp_3|^{-1}f|U_{\fp_3}\in \cH(\fn, R).$$ 
Using Lemma 2.17 and Lemma 2.22 in \cite{PW}, we have 
$$
(\phi_3|B_{\fp_3})^\ast(\fm)=
\begin{cases}
0 & \text{if }\fp_3\nmid \fm;\\
f^\ast(\fm)& \text{if }\fp_3\mid \fm. 
\end{cases}
$$ 
Therefore $\left(f-(\phi_3|B_{\fp_3})\right)^\ast(\fm)=0$ unless $\fp_1$ or $\fp_2$ divides $\fm$. 
By Lemma 2.23 and Theorem 2.24 in \cite{PW}, 
there exist $\phi_1\in \cH(\fn\fp_3/\fp_1, R)$ and $\phi_2\in \cH(\fn\fp_3/\fp_2, R)$ such that
$$
s_{\fp_1, \fp_2}\left(f-(\phi_3|B_{\fp_3})\right)=\phi_1|B_{\fp_1}+\phi_2|B_{\fp_2}. 
$$
Next, by Proposition 2.12 and Lemma 2.21 in \cite{PW}, $\phi_3|B_{\fp_3}|(W_{\fp_3}+U_{\fp_3})=(|\fp_3|+1)\phi_3$. Hence,  
$$
s_{\fp_1, \fp_2}\phi_3|B_{\fp_3}|(W_{\fp_3}+U_{\fp_3})=(|\fp_3|+1)s_{\fp_1, \fp_2}\phi_3
$$
$$
=s_{\fp_1, \fp_2}f'-\phi_1'|B_{\fp_1}-\phi_2'|B_{\fp_2},
$$
where 
$$
f'=f|(U_{\fp_3}+W_{\fp_3})\in \cH(\fn/\fp_3, R)\qquad (\text{cf. Lemma \ref{lemUW}}),
$$
$$
\phi_1'=\phi_1|(U_{\fp_3}+W_{\fp_3})\in \cH(\fn/\fp_1, R), \quad 
\phi_2'=\phi_2|(U_{\fp_3}+W_{\fp_3})\in \cH(\fn/\fp_2, R). 
$$
Let $\phi_3':=s_{\fp_1, \fp_2}f'$. Then 
$$
f'':=s_{\fp_1, \fp_2}(|\fp_3|+1)f-s_{\fp_1, \fp_2}\phi_3'|B_{\fp_3}
$$
$$
=((|\fp_3|+1)\phi_1-\phi_1'|B_{\fp_3})|B_{\fp_1}+((|\fp_3|+1)\phi_2-\phi_2'|B_{\fp_3})|B_{\fp_2}. 
$$
This implies that $f''\in \cH(\fn, R)$ has $(f'')^\ast(\fm)=0$ unless $\fp_1$ or $\fp_2$ divides $\fm$. 
Therefore, applying Theorem 2.24 in \cite{PW} one more time, we deduce that there exist $\phi_1''\in \cH(\fn/\fp_1, R)$ 
and $\phi_2''\in \cH(\fn/\fp_2, R)$ such that $s_{\fp_1, \fp_2}f''=\phi_1''|B_{\fp_1}+\phi_2''|B_{\fp_2}$. If we 
denote $\phi_3''=s_{\fp_1, \fp_2}^2\phi_3'$, then we proved 
$$
s_{\fp_1, \fp_2}^2(|\fp_3|+1)f=\phi_1''|B_{\fp_1}+\phi_2''|B_{\fp_2}+ \phi_3''|B_{\fp_3}. 
$$

We can interchange the roles of the primes $\fp_1, \fp_2, \fp_3$ and repeat the same argument to conclude that 
there exist $f_i\in \cH(\fn/\fp_i, R)$, $1\leq i\leq 3$, such that 
$$
sf=f_1|B_{\fp_1}+f_2|B_{\fp_2}+f_3|B_{\fp_3}, 
$$
where 
$$
s =\mathrm{gcd}\left(s_{\fp_1, \fp_2}^2(|\fp_3|+1), s_{\fp_1, \fp_3}^2(|\fp_2|+1),s_{\fp_2, \fp_3}^2(|\fp_1|+1)\right). 
$$
Finally, by Lemma \ref{lemgcds}, $s=s_{\fp_1, \fp_2}s_{\fp_1, \fp_3}s_{\fp_2, \fp_3}$. 
\end{proof}

\begin{lem}\label{lemgcds}
Let $m_1, m_2,m_3$ be positive integers. Put $s_{i,j}=\gcd(m_i, m_j)$ for $1\leq i<j\leq 3$. Then 
$$
\gcd(s_{1,2}^2m_3, s_{1,3}^2m_2, s_{2,3}^2m_1)=s_{1,2}s_{1,3}s_{2,3}. 
$$
\end{lem}
\begin{proof}
Let $p$ be an arbitrary prime number. Let $p^i, p^j, p^k$ be the largest powers of $p$ dividing $m_1, m_2, m_3$, respectively. 
We can assume without loss of generality that $i\leq j\leq k$. Then the largest power of $p$ dividing $s_{1,2}s_{1,3}s_{2,3}$ is $p^n$, 
where $n=i+i+j=2i+j$. On the other hand, the largest power of $p$ dividing $\gcd(s_{1,2}^2m_3, s_{1,3}^2m_2, s_{2,3}^2m_1)$ is $p^{n'}$, 
where 
$$
n'=\min(2i+k, 2i+j, 2j+i)=2i+j=n. 
$$
\end{proof}

\begin{cor}\label{corIndex}
Assume $\fn=xyz$, where $x,y,z\lhd A$ are distinct primes of degree $1$. If a prime number $\ell$ divides the index $[\T(\fn):\T(\fn)^0]$, 
then $\ell$ divides $q(q+1)$. Conversely, any prime dividing $q+1$ also divides $[\T(\fn):\T(\fn)^0]$. 
\end{cor}
\begin{proof}
Assume $\ell\nmid q(q+1)$. Consider the $\F_\ell$-linear map 
$$
\cH_{00}(\fn, \F_\ell)\to \Hom(\T(\fn)^0\otimes \F_\ell, \F_\ell)
$$
obtained from the pairing (\ref{GPairing}). If $\ell$ divides the index $[\T(\fn):\T(\fn)^0]$, then 
the kernel of the above map is non-zero. This means that there is $0\neq f\in \cH_{00}(\fn, \F_\ell)$ 
such that $f^\ast(\fm)=0$ unless $x$, $y$, or $z$ divides $\fm$. Applying Theorem \ref{thmIntAL}, we can write 
$$
f=f_1|B_x+f_2|B_y+f_3|B_z
$$
for some $f_1\in \cH(xy, \F_\ell)$, $f_2\in \cH(xz, \F_\ell)$, $f_3\in \cH(yz, \F_\ell)$. 
Moreover, it is not hard to see from the construction of $f_1, f_2, f_3$ in the proofs 
of Theorem \ref{thmIntAL} and \cite[Thm. 2.24]{PW} that we can choose these harmonic cochains 
to be from the $\cH_{00}$ part of their corresponding spaces. But $\cH_0(\fn', \Z)=0$ for any $\deg(\fn')=2$, which implies 
that $f$ must be $0$, a contradiction. 

To prove the second statement, let $\ell$ be a prime not dividing $[\T(\fn):\T(\fn)^0]$. Then $\T(\fn)\otimes \F_\ell=\T(\fn)^0\otimes \F_\ell$. 
This implies that $(\T(\fn)/\fE(\fn))\otimes \F_\ell$ is cyclic. On the other hand, if $\ell\mid (q+1)$, then 
$(\T(\fn)/\fE(\fn))\otimes \F_\ell\cong \F_\ell\times \F_\ell\times \F_\ell$ by Theorem \ref{thmTE}. 
\end{proof}

\begin{rem}
The harmonic cochains $f_1, f_2, f_3$ in Theorem \ref{thmIntAL} are not necessarily unique. For example, 
for distinct $x,y,z$ of degree $1$
$$
(E_z-E_{yz})|B_x - (E_z-E_{xz})|B_y - (E_x-E_y)|B_z =0, 
$$
where $E_\fm\in \cH(\fm, \C)$ are the Eisenstein series normalized so that the first Fourier coefficient is $1$. 
\end{rem}



\section{The rational torsion subgroup}\label{sRT}

Let $\fn\lhd A$ be a non-zero ideal and $J:=J_0(\fn)$. 
By the Lang-N\'eron theorem (see \cite{Conrad}), the group of $F$-rational points of $J$ is 
finitely generated, in particular, its torsion subgroup $\cT(\fn):=J(F)_\tor$ is finite. 
Corollary \ref{cor2.5} and Lemma~\ref{lem2.4} imply that $\cC(\fn)\subseteq \cT(\fn)$ when $\deg(\fn)=3$. 
(Also, in this case the rank of $J(F)$ is zero, so $J(F)=\cT(\fn)$.)

For a prime number $\ell$, denote by $\cT(\fn)_\ell$ the $\ell$-primary subgroup of $\cT(\fn)$. If $\ell\neq p$, then 
the Eichler-Shimura congruence relation can be used to show that $\cT(\fn)_\ell$ is Eisenstein; 
cf. \cite[Lem. 7.16]{Pal}. This implies that if ($\ell\neq p$ and $\cT(\fn)_\ell\neq 0$), then  $\ell$ is an Eisenstein prime number. 
(As far as we know, it is not known in general whether $\cT(\fn)_p$ is Eisenstein.)

\begin{thm}\label{thmRTT3}
If $\fn=T^3$, then 
$
\cT(\fn)=\cC(\fn)\cong \Z/q^2\Z$. 
\end{thm}
\begin{proof}
Since $\T(\fn)/\fE(\fn)\cong \Z/q^2\Z$, the only Eisenstein prime number is $p$, so $\cT(\fn)_\ell=0$ if $\ell\neq p$. 
We just need to 
prove that $G:=\cT(\fn)_p=\cC(\fn)$. Let $\cJ$ denote the N\'eron model of $J$ over $\cO_\infty$. 
By the extension property for \'etale points of N\'eron models, $G$ extends to a finite flat subgroup 
scheme $\cG$ of $\cJ$. Consider the connected-\'etale sequence of $\cG$: 
$$
0\to \cG^0\to \cG\to \cG^{\mathrm{et}}\to 0. 
$$
The formation of this sequence is compatible with any local base change to another henselian local ring, 
so the special fibre $\cG^0_{\F_\infty}$ is a subgroup scheme of $\cJ^0_{\F_\infty}$.  Since 
$\cJ^0_{\F_\infty}$ is isomorphic to a product of copies of the multiplicative group $\gm_{m, \F_\infty}$, 
whose $p$-primary torsion is connected, 
the scheme-theoretic intersection $\cG_{\F_\infty}\cap \cJ^0_{\F_\infty}$ is $\cG^0_{\F_\infty}$, and   
the Cartier dual of $\cG^0_{\F_\infty}$ is \'etale. On the other hand, the Cartier dual of $G$ 
is connected, since $G$ is a constant $p$-primary group scheme. This implies that the Cartier dual of $\cG^0$ has 
connected generic fibre but \'etale closed fibre, which is not possible unless $\cG^0$ is trivial. This 
implies that 
the canonical specialization map  
$$
\wp_\infty: \cT(\fn)_p\To \Phi_\infty 
$$ 
is injective. But, as we know, the restriction of $\wp_\infty$ to $\cC(\fn)_p$ is surjective, so  
$\cC(\fn)_p=\cT(\fn)_p$.  
\end{proof}

\begin{example}
Assume $q=2$. In this case $X_0(T^3)$ is an elliptic curve given by the Weierstrass equation 
$$
E: Y^2+TXY=X^3+T^2,
$$
see \cite[(9.7.3)]{GR}. One easily checks that $E(F)\cong \Z/4\Z$, generated by $(T,T)$. 
This agrees with Theorem \ref{thmRTT3}. 
\end{example}

\begin{thm}\label{thmRTT2}
If $\fn=T^2(T-1)$, then 
$
\cT(\fn)=\cC(\fn)\cong \Z/q(q^2-1)\Z$. 
\end{thm}
\begin{proof}
As in the proof of Theorem \ref{thmRTT3} we have an injection $\cT(\fn)_p\hookrightarrow \Phi_\infty$. 
Since we know that $\wp_\infty: \cC(\fn)\to \Phi_\infty$ is an isomorphism, 
$\cC(\fn)_p=\cT(\fn)_p$. Now it is enough to show that for any prime $\ell\neq p$, 
$$
\# \cT(\fn)_\ell\leq \# \cC(\fn)_\ell=\# (\Z/(q^2-1)\Z)_\ell. 
$$

The theory of Chapters 12 and 13 in \cite{KM} applied in the function field setting implies that $X_0(\fn)$ has a model 
over $\cO_{T}$ whose special fibre consists of three projective lines intersecting transversally in one point (this 
point corresponds to the supersingular Drinfeld $A$-module of rank $2$ over $\overline{\F}_T$). Two of these 
projective lines are reduced and the third one has multiplicity $q-1$. 

Adapting the methods of \cite{Edixhoven} and \cite{LorenziniTP} to this special situation, one shows that the reduction of $J$ at $T$ is 
purely additive, i.e., $\cJ^0_{\overline{\F}_T}$ is a unipotent linear group, and 
$
\#\Phi_T 
=q^2-1$.  
Since a unipotent group has no points of order $\ell\neq p$, 
the N\'eron mapping property implies that $\cT(\fn)_\ell$ maps injectively into $\Phi_T$. Therefore, 
$$
\#\cT(\fn)_\ell\leq \# (\Phi_T)_\ell=\# (\Z/(q^2-1)\Z)_\ell, 
$$
as was required to show. 
\end{proof}

\begin{rem}
The previous proof is not very satisfactory since the analogues of the results of \cite{KM}, \cite{LorenziniTP} and \cite{Edixhoven} 
for Drinfeld modular curves have not yet appeared in literature. 
\end{rem}

\begin{example}
Assume $q=2$. In this case $X_0(T^2(T-1))$ is an elliptic curve given by the Weierstrass equation 
$$
E: Y^2+TXY+TY=X^3,
$$
see \cite[(9.7.2)]{GR}. It is easy to check by elementary methods that $E(F)\cong \Z/6\Z$, 
generated by $(T,T)$. Moreover, using the Tate algorithm, one computes that the 
special fibre of $E$ at $T$ is of Kodaira type IV. This means 
that the special fibre of the minimal regular model of $E$ over $\cO_T$ consists of three 
projective lines intersecting in one point, the reduction of $E$ at $T$ is additive, and $\Phi_T\cong \Z/3\Z$. 
This agrees with Theorem \ref{thmRTT2} and the claims in its proof. 
\end{example}


\section{Characteristic as an Eisenstein prime number}\label{sEPN}

\begin{thm}\label{thm-pEis} 
Assume $\fn = \fp^2 \fm$, where $\fp \lhd A$ is prime and $\fm \lhd A$ is a proper non-zero ideal. 
$($We do not assume $\fm$ to be coprime to $\fp$.$)$
Then $p$ divides $\# \T(\fn)/\fE(\fn)$, and hence also $\# \T(\fn)^0/\fE(\fn)^0$.  
\end{thm}

We will prove the theorem in several steps. We start with a simple 
observation. Consider the following condition:  
\begin{itemize}
\item[$(\star)$] There exists $c\in \cC(\fn)$ which is Eisenstein and has order divisible by $p$.
\end{itemize}
\begin{lem}\label{lem6.2} 
If $(\star)$ holds, then $p$ divides $\# \T(\fn)/\fE(\fn)$.  
\end{lem}
\begin{proof}
By considering an appropriate multiple of $c$, we can assume $c$ has order $p$. 
Suppose on the contrary that $p$ is invertible in $\T(\fn)/\fE(\fn)$. 
Then we can find $t\in \T(\fn)$ and $e\in \fE(\fn)$ such that $1=pt+e$. Since 
$c$ is annihilated by $e$ and $p$, we get that $1\in \T(\fn)$ annihilates $c$. This is a contradiction since   
the action of $1\in \T(\fn)$ is the same as the action of $1\in \Z\subset \T(\fn)$. 
\end{proof}

Let $[0]$ and $[\infty]$ be the cusps of $X_0(\fn)$ corresponding to $\begin{pmatrix} 0 \\ 1\end{pmatrix}$ 
and $\begin{pmatrix} 1 \\ 0\end{pmatrix}$, respectively. (For $\fn$ of degree $3$ this slightly differs 
from our earlier notation, where we were denoting $[0]$ by $[1]$.) These cusps are always rational, so 
$c_0:=[0]-[\infty]$ in $J_0(\fn)$ is an $F$-rational torsion point. 

\begin{lem}\label{lem6.3}
The cuspidal divisor $c_0$ is Eisenstein. If the level $\fn$ is not square-free and is divisible 
by a prime $\fq\lhd A$ such that $\fn/\fq$ is coprime to $\fq$, then the order of $c_0$ 
is divisible by $p$. In particular, $(\star)$ holds for $\fn$. 
\end{lem}
\begin{proof}
Any upper-triangular matrix in $\GL_2(F)$ fixes $\begin{pmatrix} 1 \\ 0\end{pmatrix}\in \p^1(F)$. 
Thus, $T_\fp$ with $\fp\nmid \fn$, as a correspondence on $X_0(\fn)$, satisfies $T_\fp[\infty]=(|\fp|+1)[\infty]$. 
The Atkin-Lehner involution $W_\fn$ interchanges $[0]$ and $[\infty]$, and commutes with $T_\fp$, so 
$$
T_\fp[0]=T_\fp W_\fn[\infty]= W_\fn T_\fp[\infty]=(|\fp|+1)W_\fn[\infty]=(|\fp|+1)[0]. 
$$
This shows that $c_0$ is Eisenstein. 

Assume $\fq$ is a prime strictly dividing $\fn$. 
As follows from Theorem 5.1 and Lemma 5.2 in \cite{PW}, 
the special fibre of the minimal regular model of $X_0(\fn)$ over $\cO_\fq$ is geometrically reduced 
and consists of two geometrically irreducible components $Z, Z'$, both isomorphic to $X_0(\fn/\fq)_{\F_\fq}$, 
intersecting transversally in a certain number 
of points. Moreover, $Z, Z'$ might also be joined by a certain number of chains of projective lines of length $q+1$ as in Figure \ref{Fig2}.  
(The number of points in the intersection of $Z$ and $Z'$, as well as the number of chains of projective lines, 
can be deduced from \cite[Lem. 5.2]{PW}.) 
In other words, the special fibre $X_0(\fn)_{\overline{\F}_\fq}$ of this model looks like the Figures in Section \ref{sCDG}.  
A calculation similar to the calculation in $\S$\ref{ssCDGxyz} shows that the image of multiplication $\Phi_\fq\xrightarrow{q+1}\Phi_\fq$ 
lies in the subgroup of $\Phi_\fq$ generated by $z:=Z-Z'$. Thus, we have an equality of $p$-primary subgroups 
$(\Phi_\fq)_p=\langle z\rangle_p$. If $\fn/\fq$ is not square-free, then $(\Phi_\fq)_p\neq 0$; this follows 
from Theorem 5.3 in \cite{PW}, which gives a complete description of $\Phi_\fq$ as an abelian group. 

Now consider the image of $c_0$ under the canonical specialization $\wp_\fq$. The reductions of the cusps $[0], [\infty]$ lie  
on the components $Z, Z'$, away from the singular points of $X_0(\fn)_{\overline{\F}_\fq}$. This is 
a consequence of the reduction properties of the cusps (cf. \cite[Ch. 10]{KM}). 
The Atkin-Lehner involution $W_\fq$ interchanges the components $Z$ and $Z'$, and $W_{\fn/\fq}$ maps 
each component to itself; this follows from the modular interpretation of $X_0(\fn)_{\overline{\F}_\fq}$. 
Assume that we have labelled these two components so that $[\infty]\in Z'$.  
Since $W_\fn=W_\fq W_{\fn/\fq}$ and $W_\fn[\infty]=[0]$, we conclude that the reduction 
of $[0]$ lies on $Z$. Hence $\wp_\fq(c_0)= z$. But from the previous paragraph we know that if $\fn/\fq$ is not square-free, then $p$ divides 
the order of $z$. This implies that $p$ also divides the order of $c_0$. 
\end{proof}

\begin{lem}\label{lem6.4}
Assume $\fn=\fp^3$, where $\fp\lhd A$ is prime. There exists $c\in \cC(\fn)$ satisfying $(\star)$. 
Moreover, $c$ can be chosen to be rational over $F$.  
\end{lem}
\begin{proof} Consider the functorial morphism $X_0(\fp^3)\to X_0(\fp^2)\to X_0(\fp)\to X_0(1)$. 
The cusps of these curves and the ramification indices under the morphisms are given in Figure \ref{FigCp3}, where 
$$
m=\frac{|\fp|-1}{q-1}. 
$$ 
The ramification indices can be computed using \cite[(3.10)]{Discriminant}. 
The cusps $\tilde{z}_1, \cdots, \tilde{z}_m$ are given by $\begin{pmatrix} u \\ \fp\end{pmatrix}$, $u$ is monic of degree $<\deg(\fp)$; 
i.e., they have height $1$ in the terminology of \cite{Discriminant}. This implies that $\tilde{z}_1, \cdots, \tilde{z}_m$
are conjugate over $F$; see Lemma \ref{lemCD1}. (The cusps $\omega_1, \cdots, \omega_m$ 
are also conjugate over $F$, so only the cusps $[0]$ and $[\infty]$ of $X_0(\fp^3)$ are rational over $F$, unless $\deg(\fp)=1$.) 
\begin{figure}
$$
\xymatrix{
X_0(\fp^3) \ar[d]_{\tilde{f}} & & [0]\ar@{-}[dr]_-{|\fp|} & \tilde{z}_1 \ar@{-}[dr]_-{|\fp|} & \cdots & \tilde{z}_m \ar@{-}[dr]_-{|\fp|} 
& \omega_1\ar@{-}[drr]_-{q-1} & \cdots & \omega_m \ar@{-}[d]_-{q-1} & [\infty]\ar@{-}[dl]^-{1}\\ 
X_0(\fp^2)\ar[d]_{f} & &  & [0]\ar@{-}[dr]_-{|\fp|} & {z}_1 \ar@{-}[drr]_-{q-1} & \cdots & {z}_m \ar@{-}[d]_-{q-1}& & [\infty] \ar@{-}[dll]^-{1} & \\
X_0(\fp)\ar[d] & &  &  & [0]\ar@{-}[dr]_-{|\fp|} & & [\infty]\ar@{-}[dl]^-{1} & & &\\ 
X_0(1) & &  &  & & [\infty] &  & & &
}
$$
\caption{Cusps of $X_0(\fp^3)$}\label{FigCp3}
\end{figure}

Consider 
$$
c:=\sum_{i=1}^m (\tilde{z}_i-[0])\in \cC(\fp^3). 
$$
Note that the cusps appearing in $c$ are exactly the ones which totally ramify under $\tilde{f}$. From the 
previous discussion we see that $c$ is rational over $F$. Let $\fq\neq \fp$ be a prime. Any matrix of the form 
$\begin{pmatrix} \fq & 0 \\ 0 & 1\end{pmatrix}$ or $\begin{pmatrix} 1 & s \\ 0 & \fq\end{pmatrix}$, $s\in A$, preserves 
the heights of the cusps, and induces a permutation of the set $\{ \tilde{z}_1, \cdots, \tilde{z}_m\}$. This implies that 
$T_\fq c=(|\fq|+1)c$, so $c$ is Eisenstein. It remains to show that the order of $c$ is divisible by $p$. 

Let $\Delta(z)$, $z\in \Omega$, be the Drinfeld discriminant function.  
This is a $\C_\infty$-valued modular form for $\GL_2(A)$ of weight $q^2-1$ and type $0$; cf. \cite{Discriminant}.  
For $a\in A$, denote $\Delta_a(z)=\Delta(a z)$. The functions $\Delta/\Delta_\fp$ and $\Delta_\fp/\Delta_{\fp^2}$
are $\G_0(\fp^3)$ invariant, so can be considered as rational functions on $X_0(\fp^3)_{\C_\infty}$. One computes 
that 
\begin{equation}\label{eqDeltap}
\mathrm{div}\left(\left(\frac{\Delta_\fp}{\Delta}\right)^{|\fp|}\cdot \left(\frac{\Delta_\fp}{\Delta_{\fp^2}}\right)\right) 
= (|\fp|^2-1)(q-1)|\fp|\cdot c. 
\end{equation}
If the order of $c$ is not divisible by $p$, then the order of $c$ divides $(|\fp|^2-1)(q-1)$. This implies that 
there exists a function $\Theta\in \C_\infty(X_0(\fp^3))$ such that $\mathrm{div}(\Theta)=(|\fp|^2-1)(q-1) c$. 
Comparing with (\ref{eqDeltap}), we get 
$$
\Theta^{|\fp|}=\alpha \left(\frac{\Delta_\fp}{\Delta}\right)^{|\fp|}\cdot \left(\frac{\Delta_\fp}{\Delta_{\fp^2}}\right)
\quad \text{for some $\alpha\in \C_\infty^\times$}. 
$$
Therefore, $\frac{\Delta_\fp}{\Delta_{\fp^2}}(z)=\frac{\Delta}{\Delta_\fp}(\fp z)$ is a $|\fp|$-th power in $\cO(\Omega)^\times$ 
(= the group of nowhere vanishing holomorphic functions on $\Omega$). But according to \cite[Cor. 3.5]{Discriminant} 
the largest integer $r$ such that $\Delta/\Delta_a$ has an $r$th root in $\cO(\Omega)^\times$ divides $(q-1)(q^2-1)$. 
This leads to a contradiction. 
\end{proof}

\begin{lem}\label{lem6.5}
Let $\fm, \fn\lhd A$. Assume $\fm$ divides $\fn$. If there is $c\in \cC(\fm)$ satisfying $(\star)$, then there is 
$c'\in \cC(\fn)$ satisfying $(\star)$. Moreover, if $c$ is $F$-rational, then $c'$ also can be chosen to be $F$-rational. 
\end{lem}
\begin{proof} By Picard functoriality, the morphism $f:X_0(\fn)\to X_0(\fm)$ in (\ref{eqLLM}) induces 
a homomorphism $f^\ast: J_0(\fm)\to J_0(\fn)$ defined over $F$. Moreover, $f^\ast$ is compatible with the action of $T_\fp$ 
for any prime $\fp\nmid \fn$, and restricts to a homomorphism $\cC(\fm)\to \cC(\fn)$. 
Therefore, it is enough to show that the kernel of $f^\ast$ does not have any torsion points of 
order $p$. 

Let $\G:=\G_0(\fm)$ and $\Delta:=\G_0(\fn)$. 
Denote by $\G^\ab$ the abelianization of $\G$ and let $\bG:=\G^\ab/(\G^\ab)_\tor$ 
be the maximal abelian torsion-free quotient of $\G$. The inclusion $\Delta\hookrightarrow \G$ 
induces a homomorphism $V:\bG\to \bD$, the transfer map; see \cite[p. 71]{GR}. 
First, we note that the homomorphism $V: \bG\to \bD$ is injective with torsion-free cokernel. 
Indeed, by \cite[p. 72]{GR}, there is a commutative diagram 
$$
\xymatrix{
\bG \ar[r]^-{j_\G} \ar[d]_V &  \cH_0(\fm, \Z) \ar@{^{(}->}[d] \\
\bD \ar[r]^-{j_\Delta} & \cH_0(\fn, \Z)
}
$$
where the right vertical map is the natural injection. This last homomorphism 
obviously has torsion-free cokernel. Since by \cite{GN} $j_\G$ and $j_\Delta$ are 
isomorphisms, the claim follows.

Next, by the results in Sections 6 and 7 of \cite{GR}, there is a commutative 
diagram 
$$
\xymatrix{
0\ar[r] & \bG \ar[r] \ar[d]^V &  \Hom(\bG, \C_\infty^\times) \ar[r]\ar[d]^\phi & J_0(\fm)\ar[r]\ar[d]^{f^\ast} & 0 \\
0\ar[r] & \bD \ar[r] & \Hom(\bD, \C_\infty^\times) \ar[r] & J_0(\fn) \ar[r]& 0. 
}
$$
Since the cokernel of $V$ is a free $\Z$-module, $\ker(f^\ast)=\ker(\phi)$. Finally, 
$\phi$ is a homomorphism of tori in characteristic $p$, so the $p$-primary part of its kernel is connected. 
\end{proof}

\begin{proof}[Proof of Theorem \ref{thm-pEis}] Assume the level $\fn$ is divisible by $\fp^2$ 
for some prime $\fp$, but $\fn\neq \fp^2$. 
Thanks to Lemma \ref{lem6.2} and Lemma \ref{lem6.5}, it is enough to show that $(\star)$ 
holds for $\fn=\fp^3$ and $\fn=\fp^2\fq$, where $\fq\neq \fp$ is prime. This follows 
from Lemma \ref{lem6.3} and Lemma \ref{lem6.4}. 
\end{proof}

\begin{cor}\label{corpTors}
Assume $\fn$ is divisible by $\fp^2$ for some prime $\fp\lhd A$, but $\fn\neq \fp^2$. Then 
there is a cuspidal divisor of order divisible by $p$ which is Eisenstein and rational over $F$. 
In particular, $\cC(\fn)(F)_p\neq 0$ and $\cT(\fn)_p\neq 0$. 
\end{cor}

The only non square-free level excluded from Theorem \ref{thm-pEis} is $\fn=\fp^2$. We will show 
that this is a necessary restriction. First, we show that, in contrast to Corollary \ref{corpTors}, $\cC(\fp^2)$ does not have any $F$-rational 
points of order $p$.   

\begin{lem}\label{prop6.6}
$\cC(\fp^2)(F)_p=0$. 
\end{lem}
\begin{proof}
With notation as in Figure \ref{FigCp3}, let $c_i:=z_i-[\infty]$, $1\leq i\leq m$. As in the proof of Lemma \ref{lem6.4}, 
the cusps $z_1, \dots, z_m$ are conjugate over $F$, so $\cC(\fp^2)(F)$ is generated by $c_0=[0]-[\infty]$ 
and $c:=\sum_{i=1}^m (z_i-[\infty])$. By \cite[Cor. 3.25]{Discriminant}, the order of $c_0$ is 
$$
M(\fp)=
\begin{cases}
\frac{|\fp|^2-1}{q^2-1} & \text{if $q$ is even or $\deg(\fp)$ is odd};\\
\frac{|\fp|^2-1}{2(q^2-1)} & \text{otherwise.}
\end{cases}
$$
Thus, it is enough to show that the order of $c$ is coprime to $p$. 
By \cite[Cor. 3.23]{Discriminant}, the cuspidal divisor $([0]-[\infty])\in \cC(\fp)$ has order 
\begin{equation}\label{eqNfp}
N(\fp)=
\begin{cases}
\frac{|\fp|-1}{q-1} & \text{if $\deg(\fp)$ is odd};\\
\frac{|\fp|-1}{q^2-1} & \text{if $\deg(\fp)$ is even.}
\end{cases}
\end{equation}
On the other hand, 
$f^\ast([0]-[\infty])=|\fp|c_0-(q-1)c$. Hence $M(\fp)N(\fp)(q-1)c=0$. 
Since $M(\fp)N(\fp)(q-1)$ is coprime to $p$, the claim follows. 
\end{proof}

\begin{example}\label{exmP2}
Assume $q=2$, $\fp=T^2+T+1$, and $\fn=\fp^2$. The genus of $X_0(\fn)$ is $2$. To simplify the 
notation, we will write $\cC$ for $\cC(\fn)$ and $\T$ for $\T(\fn)$.  
The cusps of $X_0(\fn)$ are $[0]$, $[\infty]$, and 
$$
z_1=\begin{pmatrix} 1 \\ \fp\end{pmatrix}, \quad z_2=\begin{pmatrix} T \\ \fp\end{pmatrix}, 
\quad z_3=\begin{pmatrix} 1+T \\ \fp\end{pmatrix}. 
$$
In the notation of Lemma \ref{prop6.6}, $\cC$ is generated by $c_0, c_1, c_2, c_3$. 
The divisor $[0]-[\infty]$ is principal on $X_0(\fp)$, since $X_0(\fp)$ has genus $0$. Therefore
$$
0=f^\ast([0]-[\infty])=4c_0-(c_1+c_2+c_3).  
$$
On the other hand, the order of $c_0$ is $5$, so the previous relation implies that $c_1+c_2+c_3=-c_0$. Thus, 
$$
\cC(F)=\langle c_0\rangle\cong \Z/5\Z. 
$$ 

To compute the whole group $\cC$ one could use the fact that in this case $X_0(\fn)$ is hyperelliptic, and the 
Atkin-Lehner involution $W_{\fn}$ is the hyperelliptic involution; see \cite{SchweizerHE}. $W_{\fn}$ 
fixes the cusps $z_1, z_2, z_3$, and 
interchanges $[0]\leftrightarrow [\infty]$. Arguing as in Section \ref{sCDG}, i.e., pulling back from $X_0(\fn)/W_{\fn}\cong \p^1_F$
different principal divisors supported on the images of the cusps, one obtains the relations 
$$
c_0=2c_1=2c_2=2c_3. 
$$
This implies that $\cC$ is a quotient of $\Z/5\Z\times \Z/2\Z\times \Z/2\Z$. Next, one can compute the 
quotient graph $\G_0(\fn)\bs \sT$ using the algorithm in \cite{GN}. The result is given in Figure \ref{Fig10}. 
The dashed edges indicate the half-lines corresponding to the cusps, and $a_\infty$, $d_\infty$ correspond to the same 
elements of $\GL_2(\Fi)$ as in Figure \ref{Fig5}. From this one easily computes, as in Section \ref{sCDG}, that 
$\Phi_\infty\cong \Z/5\Z\times \Z/2\Z\times \Z/2\Z$ and $\wp_\infty: \cC\to \Phi_\infty$ 
is surjective. Therefore
$$
\cC\overset{\wp_\infty}{\cong}\Phi_\infty\cong  \Z/5\Z\times \Z/2\Z\times \Z/2\Z. 
$$
In particular, $\cC_p\cong \Z/2\Z\times \Z/2\Z$ is non-trivial, although $\cC(F)_p=0$.  

\begin{figure}
\begin{tikzpicture}[scale=1.4, ->, >=stealth, semithick, inner sep=.5mm, vertex/.style={circle, fill=black}]

\node[vertex] (00) at (0, 0) {};
\node[vertex] (02) at (0, 2) {};
\node[vertex] (22) at (2, 2) {};
\node[vertex] (20) at (2, 0) {};
\node[vertex] (10) at (1, 0) {};
\node[vertex] (115) at (1, 1.5) {};
\node[vertex] (125) at (1, 2.5) {};
\node[vertex] (105) at (1, -0.5) {};
\node[vertex] (v) at (.5, -.5) {};
\node (inf) at (-1, -1) {$[\infty]$};
\node (0) at (3, -1) {$[0]$};
\node (z1) at (1, -1.2) {$z_1$};
\node (z2) at (1, .8) {$z_2$};
\node (z3) at (1, 3.2) {$z_3$};

\path[]
(22) edge  (115)
(22) edge (125)
(115) edge  (02)
(125) edge  (02)
(02) edge node[auto,swap] {$a_\infty$}  (00)
(20) edge (10)
(10) edge node[auto,swap] {$d_\infty$} (00)
(22) edge (20)
(v) edge (105)
(105) edge (10)
(00) edge[dashed] (inf) 
(20) edge[dashed] (0)
(105) edge[dashed] (z1)
(115) edge[dashed] (z2)
(125) edge[dashed] (z3);

\end{tikzpicture}
\caption{$\G_0((T^2+T+1)^2)\bs \sT$}\label{Fig10}
\end{figure}

Next, we consider the Eisenstein ideal. Calculations similar to those in $\S$\ref{sPairing} show that 
$$
\T=\T^0\cong \Z T_T\oplus \Z T_{T+1},  
$$ 
and 
$$
T_T+T_{T+1}=1. 
$$
This implies $\T/\fE\cong \Z/N\Z$ is cyclic, and $N$ divides $5$ (as $T_T+T_{T+1}=1$ modulo $\fE$ 
becomes $3+3=1$). Since $c_0$ is Eisenstein of order $5$, we get 
$$
\T/\fE\cong \Z/5\Z. 
$$
Hence $\cC[\fE]\subseteq \cC[5]$, and we get 
$$
\cC[\fE]=\langle c_0\rangle=\cC(F). 
$$
Since $\wp_\infty$ is $\T$-equivariant, we also get 
$$
\Phi_\infty[\fE]\cong \cC[\fE] \cong \T/\fE\cong \Z/5\Z. 
$$
Hence neither $\Phi_\infty$ nor $\cC$ is Eisenstein, and $p$ is not an Eisenstein prime number. 

Finally, consider the rational torsion subgroup $\cT$ of $J:=J_0(\fn)$. We apply the argument in the 
proof  of Theorem \ref{thmRTT3}, which shows that there is an injection $\cT\hookrightarrow \cJ(\F_\infty)$. 
Since $\cJ^0$ is a split torus, $\cJ^0(\F_2)\cong \F_2^\times\times \F_2^\times=1$. Hence 
$\cT\hookrightarrow \Phi_\infty$. This shows that if $\cT\neq \cC(F)$, then $J$ has a rational $2$-torsion point. 
But $J$ is $2$-dimensional and the characteristic is $2$, so 
$
J[2]\subseteq  \Z/2\Z\times \Z/2\Z$. Since $\cC[2]\cong \Z/2\Z\times \Z/2\Z$, we get $J[2]=\cC[2]$. 
This is a contradiction as $\cC(F)[2]=0$. Hence $$\cT=\cC(F)\cong \Z/5\Z.$$
\end{example}

\begin{thm}\label{thmEispnot}
Let $\fp\lhd A$ be prime. Then 
$$
\T(\fp)/\fE(\fp)\cong \Phi_\infty[\fE(\fp)]\cong \Z/N(\fp)\Z, 
$$
where $N(\fp)$ is defined in (\ref{eqNfp}). In particular, $p$ is not an Eisenstein prime for level $\fp$. 
\end{thm}
\begin{proof}
By \cite[Thm. 1.2]{PalIJNT}, 
\begin{equation}\label{eqPalT0}
\T(\fp)^0/\fE(\fp)^0\cong \Z/N(\fp)\Z.
\end{equation}
(Note that what we denote by $\T(\fp)^0$ in this paper is denoted by $\T(\fp)$ in \cite{PalIJNT}.) 
The perfectness of the pairing (\ref{GPairing}) after inverting $p$, or rather the 
argument used in the proof of \cite[Thm. 3.17]{Analytical} and Corollary \ref{corIndex}, implies that 
\begin{equation}\label{eqTT0}
\T(\fp)\otimes \Z[p^{-1}]=\T(\fp)^0\otimes \Z[p^{-1}].
\end{equation}

Since $p$ does not divide $\# \T(\fp)^0/\fE(\fp)^0$, by Lemma \ref{lemTE0}, $p$ 
does not divide $\# \T(\fp)/\fE(\fp)$ either. Combining this with (\ref{eqPalT0}) and (\ref{eqTT0}), we get 
$\T(\fp)/\fE(\fp)\cong \Z/N(\fp)\Z$. 

Since $p$ is not an Eisenstein prime number, $\Phi_\infty$ cannot have Eisenstein elements of order divisible by $p$. 
Thus, after tensoring (\ref{eqmopa-inf}) with $\Z[p^{-1}]$, and again using the perfectness of (\ref{GPairing}) after inverting $p$, we 
get a surjection $\T(\fp)/\fE(\fp)\to \Phi_\infty[\fE(\fp)]$. It remains to show that $\Phi_\infty[\fE(\fp)]\supseteq \Z/N(\fp)\Z$. 

Let $J[\fE(\fp)]$ denote $J_0(\fp)(\bar{F})[\fE(\fp)]$, i.e., the subgroup of the Jacobian annihilated by $\fE(\fp)$. 
It is clear from the definitions that $J[\fE(\fp)]=J[\fE(\fp)^0]$ and $\Phi_\infty[\fE(\fp)]=\Phi_\infty[\fE(\fp)^0]$, 
where on the right hand-side we consider $J_0(\fp)(\bar{F})$ 
and $\Phi_\infty$ as $\T(\fp)^0$-modules. 
By \cite[Thm. 2.5]{PapikianMRL}, $J[\fE(\fp)^0]$ is unramified over $F$, and, as an abelian group, it is isomorphic 
to $\Z/N(\fp)\Z\times \Z/N(\fp)\Z$. If we denote by $\cJ$ the N\'eron model of $J$ over $\cO_\infty$, then we get
$$J[\fE(\fp)^0]\cong \cJ_{\F_\infty}[\fE(\fp)^0]\cong \Z/N(\fp)\Z\times \Z/N(\fp)\Z.$$
It can be deduced from \cite[p. 194]{Pal} that $\cJ^0_{\F_\infty}[\fE(\fp)^0]\cong \Z/N(\fp)\Z$ 
(in fact, this coincides with the image of the Shimura subgroup of $J_0(\fp)$). 
This last isomorphism, as well as \cite[Thm. 2.5]{PapikianMRL} used earlier, 
rely on the fact that the completion of $\T(\fp)^0$ at any prime ideal in the support of $\fE(\fp)^0$ is Gorenstein --
a rather deep property of the Hecke algebra established in \cite{Pal}. Since $\Phi_\infty\cong \cJ_{\F_\infty}/\cJ^0_{\F_\infty}$, 
we conclude that $ \Z/N(\fp)\Z \subseteq \Phi_\infty[\fE(\fp)^0]$. 
\end{proof}

\begin{example}\label{example810}
Let $q=2$ and $\fp$ be either $T^4+T^3+1$ or $T^4+T+1$, which both are irreducible over $\F_2$. 
In both cases, $\rank_\Z\T(\fp)=4$, $\T(\fp)/\fE(\fp)\cong \Z/5\Z$, $\cC(\fp)\cong \Z/5\Z$, 
and $\wp_\infty: \cC(\fp)\to \Phi_\infty$ 
is injective. However, $\Phi_\infty\cong \Z/2\Z\times \Z/80\Z$ 
for $\fp=T^4+T^3+1$ and $\Phi_\infty\cong \Z/45\Z$ for $\fp=T^4+T+1$; see \cite[(5.3.3)]{GekelerCDG}. 
Theorem \ref{thmEispnot} shows that in both cases $\Phi_\infty$ is not Eisenstein, and $\Phi_\infty[\fE(\fp)]$ coincides with $\cC(\fp)$. 
\end{example}



\end{document}